\documentclass[12pt,reqno]{amsart}
\usepackage{amssymb,mathrsfs,graphicx,extpfeil,amsmath,bbm}
\usepackage{ifthen,latexsym,float,colortbl}
\usepackage[margin=1in]{geometry}
\usepackage{caption}
\usepackage{mathptmx}
\usepackage{arydshln}
\usepackage{sidecap}
\usepackage{rotating,dsfont,stackengine}
\DeclareMathOperator*{\esssup}{ess\,sup}

\usepackage{colortbl}
\definecolor{black}{rgb}{0.0, 0.0, 0.0}
\definecolor{red}{rgb}{1.0, 0.5, 0.5}
\provideboolean{shownotes} 
\setboolean{shownotes}{true} 
\newcommand{\margnote}[1]{
\ifthenelse{\boolean{shownotes}}%
{\marginpar{\raggedright\tiny\texttt{#1}}}%
{}%
}
\newcommand{\hole}[1]{
\ifthenelse{\boolean{shownotes}}%
{\begin{center} \fbox{ \rule {.25cm}{0cm} \rule[-.1cm]{0cm}{.4cm}
\parbox{.85\textwidth}{\begin{center} \texttt{#1}\end{center}} \rule
{.25cm}{0cm}}\end{center}} {} }

\graphicspath{{pics/}}

\title[Global existence of weak solutions to a BGK model]
{Global existence of weak solutions to a BGK model relaxing to the barotropic Euler equations}

\author[Choi]{Young-Pil Choi}
\address[Young-Pil Choi]{\newline Department of Mathematics\newline
Yonsei University, 50 Yonsei-Ro, Seodaemun-Gu, Seoul 03722, Republic of Korea}
\email{ypchoi@yonsei.ac.kr}

\author[Hwang]{Byung-Hoon Hwang}
\address[Byung-Hoon Hwang]{\newline Department of Mathematics Education\newline
Sangmyung University, 20 Hongjimun 2-gil, Jongno-Gu, Seoul 03016, Republic of Korea}
\email{bhhwang@smu.ac.kr}

\numberwithin{equation}{section}

\newtheorem{theorem}{Theorem}[section]
\newtheorem{lemma}{Lemma}[section]

\newtheorem{proposition}{Proposition}[section]
\newtheorem{remark}{Remark}[section]
\newtheorem{definition}{Definition}[section]

\newcommand{\sfD}{\mathsf{D}}

\newcommand{\K}{\kappa}
\newcommand{\R}{\mathbb R}
\newcommand{\N}{\mathbb N}

\newcommand{\bbn}{\mathbb N}
\newcommand{\om}{\Omega}
\newcommand{\ls}{\lesssim}

\newcommand{\T}{\mathbb T}

\newcommand{\mc}{\mathcal C}

\newcommand{\bq}{\begin{equation}}
\newcommand{\eq}{\end{equation}}
\newcommand{\e}{\varepsilon}
\newcommand{\lt}{\left}
\newcommand{\rt}{\right}
\newcommand{\lal}{\langle}
\newcommand{\ral}{\rangle}
\newcommand{\pa}{\partial}

\newcommand{\into}{\int_\om}

\newcommand{\intr}{\int_{\R^d}}
\newcommand{\intrr}{\iint_{\R^d \times \R^d}}
\newcommand{\intor}{\iint_{\om \times \R^d}}
\newcommand{\inttr}{\iint_{\T^d \times \R^d}}

\makeatletter
\def\moverlay{\mathpalette\mov@rlay}
\def\mov@rlay#1#2{\leavevmode\vtop{%
   \baselineskip\z@skip \lineskiplimit-\maxdimen
   \ialign{\hfil$\m@th#1##$\hfil\cr#2\crcr}}}
\newcommand{\charfusion}[3][\mathord]{
    #1{\ifx#1\mathop\vphantom{#2}\fi
        \mathpalette\mov@rlay{#2\cr#3}
      }
    \ifx#1\mathop\expandafter\displaylimits\fi}
\makeatother

\begin{document}
\allowdisplaybreaks

\date{\today}

\subjclass[]{}

\begin{abstract} 
	We establish the global-in-time existence of weak solutions to a variant of the BGK model proposed by Bouchut [J. Stat. Phys., 95, (1999), 113--170] which leads to the barotropic Euler equations in the hydrodynamic limit. Our existence theory makes the quantified estimates of hydrodynamic limit from the BGK-type equations to the multi-dimensional barotropic Euler system discussed by Berthelin and Vasseur [SIAM J. Math. Anal., 36, (2005), 1807--1835] completely rigorous.

\end{abstract}

\maketitle \centerline{\date}

\tableofcontents

%
%
%
%
\section{Introduction}
 
The BGK model \cite{BGK54,W54} is a relaxation time approximation of the celebrated Boltzmann equation, which describes the time evolution of velocity distribution functions of rarefied gases based on the relaxation process towards the Maxwellian distribution. One of the important topics in the model is to study the connection between mesoscopic kinetic equations and fluid dynamical macroscopic equations so called  {\it hydrodynamic limits}. By taking into account different asymptotic regimes, it is well-known that the compressible Euler equations and the Navier--Stokes equations can be derived from the BGK model at the formal level by the Chapman--Enskog or Hilbert expansion \cite{BGL91, C88, CIP94, GS03}. In the case of the rigorous derivations, several papers have been reported so far, dealing with the Navier--Stokes--Fourier equations \cite{S03}, the linear incompressible Navier--Stokes equations \cite{B10}, and the nonlinearized compressible Euler equations and the acoustic equations \cite{B04}. For the Boltzmann equation, derivations of incompressible Navier--Stokes equations and Euler equations are established \cite{BGL93, C80, GS04,MS03, Y05}. There are also many different kinetic models, which might not be associated to microscopic descriptions, for conservation laws and balance laws \cite{B99, GM83, LPT94, PT91}. We refer to \cite{Per02, S09, V08} and references therein for the general survey of the hydrodynamic limits of kinetic theory.

Among the various kinetic models for systems of conservation laws, in the present work, we are concerned with a BGK-type kinetic equation, introduced by Bouchut \cite{B99}, relaxing to the barotropic Euler equations. To be more specific, the main purpose of the current work is to establish the global-in-time existence of weak solutions to the following kinetic equation:
\begin{equation}\label{BGK}
\pa_t f + v\cdot\nabla_x f  = M[f] - f, \qquad (x,v,t) \in \Omega \times \R^d \times \R_+
\end{equation}
subject to the initial data:
\[
f(x,v,0)= f_0(x,v), \qquad (x,v) \in \Omega \times \R^d.
\]
Here $f = f(x,v,t)$ stands for the one-particle distribution function at the phase space point $(x,v) \in \Omega \times \R^d$ and time $t \in \R_+$, where $\Omega$ is a spatial domain, either $\T^d$ or $\R^d$. The equilibrium function $M[f]$ for barotropic gas dynamics is given as
\begin{equation}\label{Maxwellian}
M [f]= \begin{cases}
\displaystyle \mathbf{1}_{|u_f-v|^d\le c_d\rho_f}\qquad &\text{for}\quad \displaystyle  \gamma=\frac{d+2}{d},\\[4mm]
\displaystyle c\left(\frac{2\gamma}{\gamma-1}\rho_f^{\gamma-1}-|v-u_f|^2\right)^{n/2}_+ \qquad &\text{for} \quad \displaystyle \gamma \in \lt(1,\frac{d+2}d\rt).
\end{cases}
\end{equation}
Here $\rho_f$ and $u_f$ denote the macroscopic density and bulk velocity, respectively:
\begin{align*}
\rho_f(x,t)&=\int_{\mathbb{R}^d} f(x,v,t)\,dv,\qquad\rho_f(x,t)u_f(x,t)=\int_{\mathbb{R}^d} vf(x,v,t)\,dv,
\end{align*}
and the constants $c_d,n,$ and $c$ are given as
\begin{align*}
c_d&=\frac{d}{|\mathbb{S}_{d-1}|},
\quad n=\frac{2}{\gamma-1}-d,\quad \mbox{and} \quad c=\left(\frac{2\gamma}{\gamma-1}\right)^{-\frac{1}{\gamma-1}}\frac{\Gamma\left(\frac{\gamma}{\gamma-1}\right)}{\pi^{\frac d2}\Gamma(\frac n2+1)},
\end{align*}
respectively, where $|\mathbb{S}_{d}|$ is the surface area of $d$-sphere embedded in dimension $d+1$, i.e., $|\mathbb{S}_d|=\frac{2\pi^{\frac{d+1}{2}}}{\Gamma\left(\frac{d+1}{2}\right)}$, and $\Gamma$ is the Gamma function. 

The BGK-type kinetic model \eqref{BGK} is constructed in \cite{B99} to study different types of hydrodynamic systems, especially the barotropic gas dynamics. Note that the equilibrium function $M$ satisfies that for $\gamma \in (1, \frac{d+2}d]$,
\bq\label{moment_comp}
	\int_{\mathbb{R}^d}(1,v,|v|^2)  M[f]\,dv =
	\left(\rho_f,\rho_fu_f, 
	\rho_f |u_f|^2 +dC_d \rho_f^\gamma \right)
\eq
for some $C_d > 0$ (see Lemma \ref{moments} for detailed computations). In particular, the first two moments estimates imply
\bq\label{cons}
\intr (M[f]-f)\,dv=0\quad \mbox{and} \quad \intr v(M[f]-f)\,dv=0.
\eq
Thus it formally leads to the conservation laws of mass and momentum for \eqref{BGK}. 

The kinetic entropy associated to the equation \eqref{BGK} is given as
\bq\label{k_entropy}
H(f,v)= \begin{cases}
\displaystyle \frac{|v|^2}{2}f\qquad &\text{for}\quad \displaystyle \gamma=\frac{d+2}{d},\\[4mm]
\displaystyle \frac{|v|^2}{2}f+\frac{1}{2c^{2/n}}\frac{f^{1+2/n}}{1+2/n} \qquad &\text{for} \quad \displaystyle \gamma \in \lt(1,\frac{d+2}d\rt).
\end{cases}
\eq
It is observed in \cite{B99} that for any $f$ satisfying 
\[
\intr (f + H(f,v))\,dv < \infty, 
\]
the following  minimization principle holds: 
\begin{equation}\label{minimization}
\intr H(M[f],v)\,dv \le \intr H(f,v)\,dv.
\end{equation}

In the mono-dimensional case, i.e. $d=1$, the global-in-time existence of weak solutions and its hydrodynamic limit of the kinetic equation \eqref{BGK} are studied by Berthelin and Bouchut \cite{BB00, BB02, BB02_3}. When it comes to the multi-dimensional case, the hydrodynamic limit from the BGK-type kinetic equations to the isentropic gas dynamics is investigated by Berthelin and Vasseur \cite{BV05} based on the relative entropy method, also often called as modulated energy method. This method requires the strong regularity of solutions to the limiting system, the barotropic Euler system, thus the hydrodynamic limit is valid only before shocks appear. In \cite{BV05}, it is assumed that there exist $L^1$-solutions for $f$ satisfying the kinetic entropy inequality. However, to our best knowledge, the global-in-time existence of such solutions to \eqref{BGK} has not been established yet except the mono-dimensional case. The main purpose of this study is, therefore, to develop an existence theory for the BGK-type kinetic equation \eqref{BGK}. For the purpose of studying the hydrodynamic limit, we also need to have the constructed solutions satisfying the kinetic entropy inequality. 

\subsection{Formal derivation of the barotropic Euler system} Let us briefly and formally explain on the connection between the kinetic equation \eqref{BGK} and the barotropic Euler system. Considering the typical Euler scaling, $(t,x) \mapsto (\frac t\e, \frac x\e)$ with the relaxation parameter $\e>0$, we obtain from \eqref{BGK} that  
\bq\label{formal}
\pa_t f_\e + v\cdot\nabla_x f_\e  = \frac1\e (M[f_\e] - f_\e).
\eq
By taking into account the local moments and using \eqref{cons}, we can derive a system of local balanced laws:
\begin{align*}
&\pa_t \rho_{f_\e} + \nabla_x \cdot (\rho_{f_\e}  u_{f_\e}) = 0,\cr
&\pa_t (\rho_{f_\e}  u_{f_\e}) + \nabla_x \cdot (\rho_{f_\e}  u_{f_\e} \otimes  u_{f_\e}) + \nabla_x \cdot \lt(\intr ( u_{f_\e} - v)\otimes ( u_{f_\e} - v) f_\e\,dv\rt) =0.
\end{align*}
Note that the above system is not closed. On the other hand, if we have $\rho_{f_\e} \to \rho$ and $u_{f_\e} \to u$ as $\e \to 0$, then formally it follows from \eqref{formal} that 
\[
f^\e \to M_{\rho, u} \quad \mbox{as} \quad \e \to 0, 
\]
where
\[
 M_{\rho, u} := \begin{cases}
\displaystyle \mathbf{1}_{|u-v|^d\le c_d\rho}\qquad &\text{for}\quad \displaystyle  \gamma=\frac{d+2}{d},\\[4mm]
\displaystyle c\left(\frac{2\gamma}{\gamma-1}\rho^{\gamma-1}-|v-u|^2\right)^{n/2}_+\qquad &\text{for} \quad \displaystyle \gamma \in \lt(1,\frac{d+2}d\rt).
\end{cases}
\]
This implies
\[
\intr ( u_{f_\e} - v)\otimes ( u_{f_\e} - v) f_\e\,dv \to C_d \rho^\gamma \mathbb{I}_{d \times d}
\]
for some $C_d>0$ as $\e \to 0$ due to \eqref{moment_comp}.  Here $\mathbb{I}_{d \times d}$ denotes the $d \times d$ identity matrix. Thus at the formal level we derive the following barotropic Euler system from \eqref{formal} as $\e \to 0$:
\begin{align}\label{Euler}
\begin{aligned}
&\pa_t \rho  + \nabla_x \cdot (\rho   u ) = 0,\cr
&\pa_t (\rho   u ) + \nabla_x \cdot (\rho   u  \otimes  u ) + C_d\nabla_x \rho^\gamma  =0.
\end{aligned}
\end{align}
Note that the entropy for the above system is given by
\[
\eta(\rho, u) = \frac12 \rho |u|^2 + \frac{C_d}{\gamma-1}\rho^\gamma, \qquad \gamma \in \lt( 1, \frac{d+2}d\rt],
\]
and it follows from \cite{B99} that 
\[
\intr H(M[f],v)\,dv = \eta(\rho_f, u_f),
\]
where the kinetic entropy $H$ is appeared in \eqref{k_entropy}.

\subsection{Main result}

In order to state our main theorem, we first introduce a notion of weak solutions to the equation \eqref{BGK}.

\begin{definition}\label{def_weak} For a given $T>0$, we say that $f$ is a weak solution to \eqref{BGK} if the following conditions are satisfied: 
	\begin{itemize}
		\item[(i)] $f$ satisfies
		\[
		\left\{ \begin{array}{ll}
\displaystyle f \in L^\infty(0,T; L^1_+ \cap L^\infty(\om \times \R^d)), & \textrm{if } \displaystyle  \gamma=\frac{d+2}{d}\\[3mm]
\displaystyle f \in L^\infty(0,T; L^1_+ \cap L^{1+\frac{2}{n}}(\om \times \R^d)) & \textrm{otherwise}
  \end{array} \right.,
\]
		\item[(ii)] for all $\varphi \in \mc^1_c(\Omega \times \R^d \times [0,T])$ with $\varphi(x,v,T) = 0$,
		\begin{align*}
		- \intor f_0 \varphi_0\,dxdv - \int_0^T \intor f (\pa_t \varphi + v \cdot \nabla_x \varphi)\,dxdvdt = \int_0^T \intor \lt(M[f] - f\rt)\varphi\,dxdvdt.
		\end{align*}
	\end{itemize}
\end{definition}

\begin{theorem}\label{main_thm} Let $T>0$, and assume $\gamma \in (1,3]$ when $d=1$ and $\gamma \in (1,\frac{d+4}{d+2}]\cup \{\frac{d+2}{d}\}$ when $d\geq 2$. Suppose that the initial data $f_0$ satisfies
\[
\left\{ \begin{array}{ll}
\displaystyle f_0 \in L^1_+ \cap L^\infty(\om \times \R^d), \quad  (|x|^2+|v|^2)f_0  \in L^1(\om \times \R^d) & \textrm{if } \displaystyle  \gamma=\frac{d+2}{d}\\[3mm]
\displaystyle f_0  \in L^1_+\cap L^{1+\frac 2n}(\T^d \times \R^d),\quad |v|^2f_0  \in L^1(\T^d \times \R^d) & \textrm{otherwise}
  \end{array} \right..
\]	
Then there exists at least one weak solution $f$ to the equation \eqref{BGK} in the sense of Definition \ref{def_weak} satisfying
	\[
\left\{ \begin{array}{ll}
\displaystyle\sup_{0 \leq t \leq T}\|f(\cdot,\cdot,t)\|_{L^1\cap L^\infty} \leq \|f_0\|_{L^1 \cap L^\infty} +1 & \textrm{if } \displaystyle  \gamma=\frac{d+2}{d}\\[3mm]
\displaystyle\sup_{0 \leq t \leq T}\|f(\cdot,\cdot,t)\|_{L^1\cap L^{1+\frac 2n}} \leq \|f_0\|_{L^1\cap L^{1+\frac 2n} }+\||v|^2f_0\|_{L^1}  & \textrm{otherwise}
  \end{array} \right.
\]	
and the kinetic entropy inequality:
	\bq\label{kin_ineq}
\intor H(f,v)
\,dxdv + \int_0^t \intor H(f,v)
-H(M[f],v)
\,dxdvds\leq \intor H(f_0,v)
\,dxdv 
\eq
	for $t \in [0,T]$.
\end{theorem}

\subsection{Remarks}

Several remarks regarding Theorem \ref{main_thm} are in order:
\begin{itemize}
\item[(i)] 	In the present work, the spatial domain $\om$ is differently chosen depending on $\gamma$ as
\[
\displaystyle  \om = \left\{ \begin{array}{ll}
\displaystyle \R^d \mbox{ or }\T^d & \textrm{if } \displaystyle  \gamma=\frac{d+2}{d}\\[3mm]
\displaystyle  \T^d & \textrm{otherwise}
  \end{array} \right..
\]
Due to some technical difficulties, we were not able to consider the whole space $\om =\R^d$ in the case of $\gamma \in (1, \frac{d+2}{d})$. Precisely, in our strategy, the lower bound on $\rho$ is required to obtain the Lipschitz continuity of the equilibrium function $M$ in some weighted function space (see Section \ref{sec_idea} for details) in that case. On the other hand, the density has a finite mass, thus the boundedness of the spatial domain is indispensable.  If we consider a bounded spatial domain with appropriate boundary conditions, then one may use our idea of proof to construct the global-in-time existence of weak solutions in the case of $\gamma \in (1, \frac{d+2}{d})$.

\item[(ii)] As mentioned above, the equilibrium function $M$ is defined differently depending on $\gamma$, which results in different functional spaces in which the solution belongs to. For $\gamma=\frac{d+2}{d}$, we have the uniform  $L^1\cap L^\infty$-estimate, which enables us to make use of the relationship between macroscopic observable quantities and the kinetic energy (see Lemma \ref{lem_tech}). This, together with the velocity averaging lemma developed in \cite{KMT13}, plays a key role in proving the main result, Theorem \ref{main_thm}. In the other cases, however, the $L^\infty$-estimate is missing, so we apply the Dunford-Pettis theorem and the averaging lemma \cite{GLPS88,Per89} instead to complete the proof. 
\item[(iii)]  In the mono-dimensional case, the global-in-time existence of $L^1$-solution $f$ satisfying the entropy inequality \eqref{kin_ineq} is studied in \cite{BB00, BB02} for $\gamma \in (1,3)$. To the extent of our knowledge, our main theorem, Theorem \ref{main_thm}, provides for the first time the global existence theory for the equation \eqref{BGK} when $\gamma = 3$ in one dimension. 
\item[(iv)] As mentioned above, the hydrodynamic limit from the BGK-type kinetic equation \eqref{formal} to the multi-dimensional barotropic Euler system \eqref{Euler} is obtained in \cite{BV05} under the assumption on the existence of weak solutions to \eqref{formal} satisfying the kinetic entropy inequality. Thus, our existence theory, Theorem \ref{main_thm}, makes the results of \cite{BV05} completely rigorous for $\gamma \in (1,3]$ when $d=1$ and $\gamma \in (1,\frac{d+4}{d+2}]\cup \{\frac{d+2}{d}\}$ when $d\geq 2$.
\item[(v)] For the derivation of \eqref{Euler} with the isothermal pressure law, i.e. $\gamma = 1$, the following nonlinear Vlasov--Fokker--Planck equation can be considered:
\[
\pa_t f_\e + v\cdot\nabla_x f_\e  = \frac1\e \nabla_v \cdot (\nabla_v f_\e - (u_{f_\e} - v) f_\e).
\]
For the above equation, the global-in-time existence of weak solutions and the hydrodynamic limit are studied in \cite{BV05, KMT13}. The global existence and uniqueness of classical solutions near the global Maxwellian and its time-asymptotic behavior are also investigated in \cite{C16}.
\item[(vi)] We were not able to cover the whole range of $\gamma$, $(1, \frac{d+2}d]$ when $d \geq 2$. The reason lies in our inability to handle some singularities arising from boundaries of the supports of the equilibrium function $M$. We give more detailed explanations on that in Remark \ref{rmk_gap} below.
\end{itemize}

 \subsection{Main difficulties and strategy of the proof}\label{sec_idea}

One of the main ingredients in the proof of Theorem \ref{main_thm} is to obtain the Cauchy estimates for the approximation sequence associated to \eqref{BGK} in a proper functional space. Thus, it is necessary to control the equilibrium function $M[f]$ with macroscopic observable quantities $\rho_f$ and $u_f$. Related results on other BGK-type models can be found in \cite{Per89,PY16} in which the equilibrium function is dealt with in a weighted $L^1$-space by $(1+|v|^2)$ and \cite{CLY21,CY20,CY20 2}, where the $L^\infty$-space weighted by $(1+|v|^{q})$ is employed instead. In the current work, the situation is quite different compared to the previous works since our equilibrium function takes the form of indicator or positive functions \eqref{Maxwellian}, which does not allow us to apply any useful tools, such as mean value theorem typically used for BGK models \cite{CLY21,PP93,PY16,Yun15}. Thus, to get an estimate of the equilibrium function, we inevitably have to do some explicit calculations. 

For controlling our equilibrium functions, we notice that the $L^1$-space is quite appropriate to extract an information from the support of   $M[f]$, and the weight function $(1+|v|^2)$ is required due to the presence of the bulk velocity $u_f$. For this reason, we introduce a weighted space $L^1_2(\om\times \R^d )$ equipped with the norm:
\[
\|f\|_{L^1_2}:=\intor (1+|v|^2)f(x,v)\,dxdv.
\]
We employ the space $L^1_2(\om\times \R^d )$ to obtain the Lipschitz continuity for the equilibrium function $M[f]$. For this, it is required to calculate the integral over the symmetric difference of supports of $M[f]$ and $M[g]$. Here, we shall give a brief idea of showing the Lipschitz continuity for the equilibrium function $M[f]$ in the case $\gamma=\frac{d+2}{d}$ and $d \geq 2$. Note that 
\begin{align*}
\|f-g\|_{L^1_2} &=\intor (1+|v|^2) |M[f]-M[g]|\,dxdv \cr
&= \into \lt(\int_{\mathsf{supp}(M[f]) \cup \mathsf{supp}(M[g])} (1+|v|^2) |M[f]-M[g]|\,dv \rt)dx.
\end{align*}
Since the symmetric difference of the supports of  $M[f]$ and $M[g]$ is determined depending on the relation between macroscopic fields of $f$ and $g$, we split the velocity-domain $\mathbb{R}^d$ into four parts: $\R^d = \cup_{i=1}^4 \sfD_i$, where
\begin{align*}
\sfD_1&:=\left\{v\in\mathbb{R}^d : |u_f-u_g|> r_f+r_g  \right\},\cr
\sfD_2&:=\left\{v\in\mathbb{R}^d :  \left|r_f-r_g \right|\le |u_f-u_g|\le r_f+r_g\quad \text{and}\quad |u_f-u_g|^2> \left|r_f^2-r_g^2 \right|  \right\},\cr
\sfD_3&:=\left\{v\in\mathbb{R}^d : |r_f-r_g|\le |u_f-u_g|\le r_f+r_g\quad \text{and}\quad |u_f-u_g|^2\le\left|r_f^2-r_g^2\right|  \right\}, \quad \mbox{and}\cr
\sfD_4&:=\left\{v\in\mathbb{R}^d :  |u_f-u_g|\le |r_f-r_g| \right\}.
\end{align*}
Here for $h \in \{f,g\}$, $r_h$ and $u_h$ represent the radius and center of the support of $M[h]$, respectively. Since the support of $M[h]$ is the $d$-ball centered at $u_h$ of radius $r_h$, $\sfD_1$ indicates the case that the two $d$-balls do not intersect, and $\sfD_4$ denotes the case where the small one is completely contained within the large one, thus  $\sfD_1$ and $\sfD_4$ are rather easy to handle. The major difficulties arise when the two $d$-balls partially intersect, and to resolve the trouble, we shall split that case into $\sfD_2$ and $\sfD_3$. For reader's convenience, we illustrate the domains of $\sfD_2$ and $\sfD_3$ in Fig. \ref{domains}. Obviously, it is necessary not to lose any information on differences between macroscopic observable quantities of $f$ and $g$. In the analysis, it is important to cancel several problematic terms out instead of controlling them. Thus a careful and delicate analysis is required to deal with the symmetric difference of cases $\sfD_2$ and $\sfD_3$. For instance, to investigate the $L^1$-estimate of symmetric difference, we make use of a spherical cap $\mathbb{V}_f$ (see Fig. \ref{Spherical cap}) instead, which is calculated as
$$
|\mathbb{V}_f|=\frac{|\mathbb{S}_{d-2}|}{|\mathbb{S}_{d-1}|}\rho_f\left(-\frac{1}{d-1}(\sin^{d-1}\theta_f)\cos\theta_f +\int_0^{\theta_f}\sin^{d-2}\theta\,d\theta\right),
$$
with
$$
\theta_f=\arccos \left(\frac{r_f^2-r_g^2+|u_f-u_g|^2}{2r_f |u_f-u_g|}\right).
$$
Note that the information needed is inherent in $\theta_f$, and all constants must be preserved to match other terms that need to be removed. Thus it is necessary to identify the last term of $\mathbb{V}_f$ exactly, and for this we employ the following relation: 
$$
\int \sin^d x\,dx=-\frac 1d(\sin^{d-1}x)\cos x+\frac{d-1}{d}\int \sin^{d-2}x\,dx.
$$
Here, detailed calculations are performed differently  depending on whether the dimension value $d$ is even or odd. In addition,  when it comes to the $L^1$-estimate weighted by $|v|^2$, the analysis becomes more complicated since we have to extract the information by dealing with the following two terms together
$$
\int_{\mathbb{V}_f}u_f\cdot v\,dv+\int_{\mathbb{V}_g}u_g\cdot v\,dv.
$$
To overcome this difficulty, we apply  a rotation matrix \eqref{rotation} which measures the polar angle of the spherical coordinate system from the $v_d$-axis.  In this procedure, since the inner product is invariant under rotations and spherical caps have a symmetric structure, the calculation becomes much  easier to figure out as follows 
$$
\int_{\mathbb{V}_f}u_f\cdot v\,dv =(Ru_f)\cdot\left(0,0,\cdots,0,1   \right)|\mathbb{S}_{d-2}|\int_0^{\theta_f}\int_{\frac{\cos\theta_f }{\cos\theta}r_f}^{r_f} r^{d}\cos\theta(\sin^{d-2}\theta)  \,drd\theta.
$$
Moreover, since $\mathbb{V}_g$ is positioned in the opposite direction to the region $\mathbb{V}_f$, one finds
$$
\int_{\mathbb{V}_f}u_f\cdot v\,dv =(Ru_f)\cdot\left(0,0,\cdots,0,1   \right)|\mathbb{S}_{d-2}|\int_{\pi-\theta_g}^{\pi}\int_{\frac{\cos\theta_g }{\cos\theta}r_g}^{r_g} r^{d}\cos\theta(\sin^{d-2}\theta)  \,drd\theta
$$
in the same manner. Such geometric relation gives the negative sign between the above target terms, and this enables us to obtain the Lipschitz continuity of $M$. See the proof of Lemma \ref{lipschitz} below for more details. 
 \begin{figure}[!h]
 	\begin{center}
 		\includegraphics[scale=0.45]{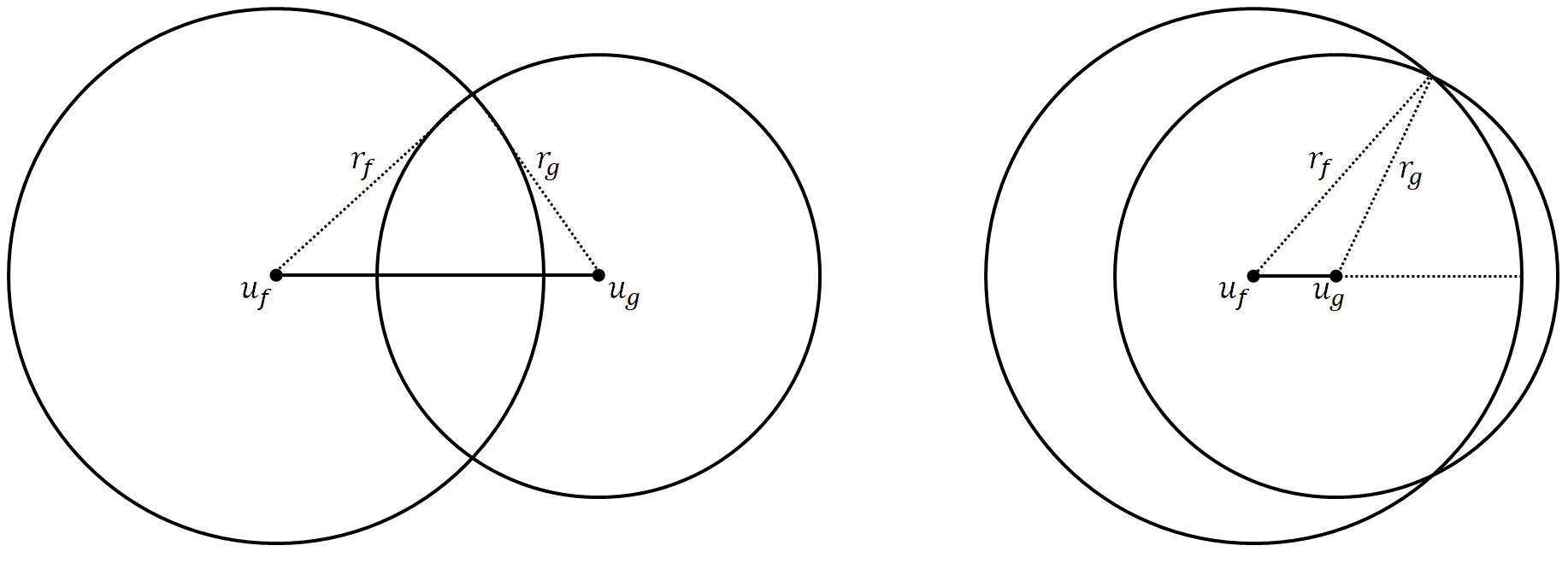}
\caption{
Illustrations of the domains $\sfD_2$ \& $\sfD_3$
}
\label{domains}
 	\end{center}
 \end{figure} 
We also would like to mention that the bounds of macroscopic quantities $\rho_f$ and $u_f$ should be assumed to control for unnecessary terms resulting from the structure of the support of $M[f]$, or the presence of the weight $|v|^2$. In the case of classical BGK model and its variants, where the equilibrium function has the form of the Maxwellian, this matter can be handled by establishing the existence of mild solutions with specific solution space containing the bounds of $\rho_f$ and $u_f$, see \cite{PP93,PY16,Yun15} for related works.  For this, weighted $L^\infty$ estimates for macroscopic fields provided by Perthame and Pulvirenti \cite{PP93} play an important role in closing the iteration scheme on the solution space. However, it does not work at all in the present work due to the totally different structure of the equilibrium function. To resolve that technical difficulty, we consider the regularized equation of \eqref{BGK} and set the regularized macroscopic fields as
\[
\rho^\e_{f_\e} := \frac{\rho_{f_\e} * \theta^\e}{1 + \e^{d+1} \rho_{f_\e} * \theta^\e} \quad \mbox{and} \quad u^\e_{f_\e} := \frac{(\rho_{f_\e} u_{f_\e})*\theta^\e}{\rho_{f_\e} * \theta^\e + \e^{2d+1}(1 + |(\rho_{f_\e} u_{f_\e})*\theta^\e|^2)},
\]
where $0<\e \le 1$ denotes the regularization parameter and $\theta^\e$ the mollifier (see \eqref{reg_eqn}). Then, we can control the regularized equilibrium function $M^\e[f_\e]$ thanks to the $\e$-dependent bounds of $\rho^\e_{f_\e}$ and $u^\e_{f_\e}$. This enables us to derive the Cauchy estimates for the approximation sequence associated to the regularized equation by using Lemma \ref{lipschitz} below. After establishing the weak solutions to the regularized equation, we obtain several bound estimates for the solution to the regularized equation uniformly in $\e$ and apply appropriate weak and strong compactness arguments in order to pass to the limit $\e\to 0$. Finally, we show that the limiting function is indeed the solution to \eqref{BGK} in the sense of Definition \ref{def_weak}.

\begin{remark}\label{rmk_gap}
For $1<\gamma <\frac{d+2}{d}$, the equilibrium function $M[f]$ reads
\begin{equation*}
M [f]=  
c\left(\frac{2\gamma}{\gamma-1}\rho_f^{\gamma-1}-|v-u_f|^2\right)^{n/2} \mathbf{1}_{|v-u_f|^2\le \frac{2\gamma}{\gamma-1}\rho_f^{\gamma-1}}.
\end{equation*} 
Note that when we apply the mean value theorem, a transitional term takes the form of
$$
\int_0^1  \left\{\theta\left(r_f^2-|v-u_f|^2 \right)+(1-\theta)\left(r_g^2-|v-u_g|^2\right) \right\}^{\frac n2-1} d\theta,
$$
which blows up at the intersection of boundaries of supports for $M[f]$ and $M[g]$. To avoid the singularity, we restrict ourselves to the case $1<\gamma \le \frac{d+4}{d+2}$, which corresponds to the case  $n\ge2$. 
\end{remark}

\subsection{Organization of the paper}
The rest of this paper is organized as follows. In Section \ref{sec_cau}, the Lipschitz continuity of $M[f]$ is investigated in a weighted space $L^1_2(\om\times \R^d)$. Here, we only deal with the case of end point $\gamma=\frac{d+2}{d}$ and postpone the other cases to Appendix \ref{app_lem} for better readability.  We introduce the regularized equation of \eqref{BGK} in the case of $\gamma=\frac{d+2}{d}$ and prove the main result, Theorem \ref{main_thm} in Section \ref{sec_bcase}. Finally, Section \ref{sec_4} is
devoted to the proof of Theorem \ref{main_thm} in the case $\gamma\in (1,3)$ when $d=1$, and $\gamma \in (1,\frac{d+4}{d+2}]$ when $d\ge 2$.

%
%
%
%

\section{Lipschitz estimate of the equilibrium function in $L^1_2(\om \times \R^d)$}\label{sec_cau}
The aim of this section is to show the Lipschitz continuity of $M[f]$ in $L^1_2(\om \times \R^d)$. More precisely, we prove that $L^1_2(\om \times \R^d)$-norm of $M[f]-M[g]$ is bounded by differences between macroscopic fields of $f$ and $g$. As mentioned in Introduction, the Lipschitz estimate will play a crucial role in constructing solutions to the regularized equation of \eqref{BGK} later in Section \ref{sec_reg}.

\begin{lemma}\label{lipschitz}
	Let the equilibrium function $M[f]$ be given by \eqref{Maxwellian}  and consider two cases:  $\gamma\in (1,3]$ when $d=1$, and  $\gamma \in (1,\frac{d+4}{d+2}] \cup\{\frac{d+2}{d}\}$ when $d\ge 2$. 
	\begin{itemize}
	\item[(i)] In the case $\gamma=\frac{d+2}{d}$ with $d\ge 1$, suppose that there exist positive constants $C_1$ and $C_2$  such that
	\begin{equation}\label{bounds}
\rho_h \le C_1 \quad \mbox{and} \quad |u_h| \le C_2, \quad h \in \{f, g\}.
	\end{equation}
	\item[(ii)] For $\gamma\in (1,3]$ with $d=1$ or $\gamma \in (1,\frac{d+4}{d+2}] $ with $d\ge 2$, we assume that there exist positive constants $C_0$, $C_1$, and $C_2$  such that
\begin{equation}\label{bounds2}
C_0\le \rho_h \le C_1 \quad \mbox{and} \quad |u_h| \le C_2, \quad h \in \{f, g\}.
\end{equation}
\end{itemize}
Then we have
	$$
	\intor (1+|v|^2) |M[f]-M[g]|\,dxdv\le C\into |\rho_f-\rho_g |+|u_f-u_g | \,dx
	$$
	for some $C>0$ which depends only on $d$, $\gamma$, and $C_i, i=0,1,2$.
\end{lemma}

As stated before, the equilibrium function is different according to the value $\gamma$. We prove Lemma \ref{lipschitz} by dividing into two cases: $\gamma = \frac{d+2}{d}$ and the others.  Both are rather lengthy and technical, thus for smoothness of reading, we only provide the details of the proof of Lemma \ref{lipschitz} in the case $\gamma = \frac{d+2}d$ here. In this case, we recall that the equilibrium function $M[f]$ reads
\[
M [f]=\mathbf{1}_{|u_f-v|^d\le c_d\rho_f} \quad \mbox{with} \quad c_d=\frac{d}{|\mathbb{S}_{d-1}|}.
\]
We postpone the proof of Lemma \ref{lipschitz} in the other cases of $\gamma$  to Appendix \ref{app_lem}.

\subsection{Proof of Lemma \ref{lipschitz} in the case $d=1$ and $\gamma=3$}

In order to explain the main ideas behind our strategy in higher dimensions, we begin with the one-dimensional case.

For the one-dimensional case, we decompose the integral as follows:
\begin{align*}
\begin{split}
\int_{\mathbb{R}}(1+v^2)\left|M[f]-M[g]\right| dv&=\int_{\mathbb{R}}(1+v^2)\left|\mathbf{1}_{|u_f-v|\le \frac{1}{2}\rho_f}-\mathbf{1}_{|u_g-v|\le \frac{1}{2}\rho_g}\right| dv\cr
&= \lt(\int_{\sfD_1} + \int_{\sfD_2} + \int_{\sfD_3}\rt)(1+v^2)\left|\mathbf{1}_{|u_f-v|\le \frac{1}{2}\rho_f}-\mathbf{1}_{|u_g-v|\le \frac{1}{2}\rho_g}\right| dv\cr
&=: I + II + III,
\end{split}
\end{align*}
where
\begin{align*}\begin{split}
\sfD_1&:=\left\{v\in\mathbb{R}\ :\ |u_f-u_g|> \frac{\rho_f+\rho_g}{2}  \right\},\cr
\sfD_2&:=\left\{v\in\mathbb{R}\ :\ \frac{|\rho_f-\rho_g|}{2}< |u_f-u_g|\le \frac{\rho_f+\rho_g}{2} \right\}, \quad \mbox{and}\cr
\sfD_3&:=\left\{v\in\mathbb{R}\ : \ |u_f-u_g|\le \frac{|\rho_f-\rho_g|}{2} \right\}.
\end{split}\end{align*}
By symmetry, we only prove the case $u_f \le u_g$. 

\vspace{.2cm}

	$\bullet$ Estimate of $I$: On the domain $\sfD_1$, the supports of the indicator functions $M[f]$ and $M[g]$ do not intersect with each other. Thus, we obtain from \eqref{moment_comp} that 
	\begin{align*}
	I&=\int_{\mathbb{R}} (1+v^2)M[f]\,dv+\int_{\mathbb{R}} (1+v^2)M[g]\,dv		\cr
	&=\left(\rho_f+\rho_fu_f^2+\frac{1}{12}\rho_f^3\right)+\left(\rho_g+\rho_gu_g^2+\frac{1}{12}\rho_g^3\right)\cr
	&=\rho_f+\rho_g+ \left(\rho_f+\rho_g\right)|u_f|^2+\rho_g \left(u_g^2-u_f^2\right) +\frac{1}{12}(\rho_f^3+\rho_g^3)
	\end{align*}
	Since $\rho_h$ and $u_h$ $(h \in \{f, g\})$ are bounded, by using the condition of $\sfD_1$, we deduce
	\[
	I \leq  C |u_f-u_g|
	\]
	for some $C>0$ which depends only on $C_i,i=1,2$.
	
	\vspace{.2cm}
	
	$\bullet$ Estimate of $II$: In this case, the supports of $M[f]$ and $M[g]$ partially intersect with each other, thus we need to estimate the symmetric difference of them. Direct computations yield
	\begin{align*}
	II &=\int_{\mathbb{R}} (1+v^2)M[f]\,dv+\int_{\mathbb{R}} (1+v^2)M[g]\,dv-2\int_{u_g-\frac{1}{2}\rho_g}^{u_f+\frac{1}{2}\rho_f} (1+v^2)\,dv\cr
	&=\rho_f+\frac13\left\{\left(u_f+\frac{\rho_f}{2}\right)^3-\left(u_f-\frac{\rho_f}{2}\right)^3 \right\}+\rho_g+\frac13\left\{\left(u_g+\frac{\rho_g}{2}\right)^3-\left(u_g-\frac{\rho_g}{2}\right)^3 \right\}\cr
	&\quad -2\left(u_f-u_g+\frac{\rho_f+\rho_g}{2}\right)-\frac 23\left\{ \left(u_f+\frac{\rho_f}{2}\right)^3-\left(u_g-\frac{\rho_g}{2}\right)^3 \right\}\cr
	&=2\left(u_g-u_f\right)+\frac{1}{3}\left\{\left(u_g+\frac{\rho_g}{2}\right)^3-\left(u_f+\frac{\rho_f}{2}\right)^3+\left(u_g-\frac{\rho_g}{2}\right)^3-\left(u_f-\frac{\rho_f}{2}\right)^3\right\}.
	\end{align*}
We then further estimate the second term on the right hand side of the above as
	\begin{align*}
	&\frac{1}{3}\left\{\left(u_g+\frac{\rho_g}{2}\right)^3-\left(u_f+\frac{\rho_f}{2}\right)^3+\left(u_g-\frac{\rho_g}{2}\right)^3-\left(u_f-\frac{\rho_f}{2}\right)^3\right\}\cr
	&\quad =\frac{1}{3}\left(u_g-u_f+\frac{\rho_g-\rho_f}{2}\right)\left\{\left(u_g+\frac{\rho_g}{2}\right)^2+\left(u_g+\frac{\rho_g}{2}\right)\left(u_f+\frac{\rho_g}{2}\right)+\left(u_f+\frac{\rho_f}{2}\right)^2\right\}\cr
	&\qquad +\frac{1}{3}\left(u_g-u_f-\frac{\rho_g-\rho_f}{2}\right)\left\{\left(u_g-\frac{\rho_g}{2}\right)^2+\left(u_g-\frac{\rho_g}{2}\right)\left(u_f-\frac{\rho_f}{2}\right)+\left(u_f-\frac{\rho_f}{2}\right)^2\right\}.
	\end{align*}
	This together with \eqref{bounds} and the condition of $\sfD_2$ gives
	\begin{align*}
	II &\le  2\left|u_f-u_g\right|+ \frac{1}{3}\left(|u_g-u_f|+\frac{|\rho_f-\rho_g|}{2}\right)\left|\left(u_g+\frac{\rho_g}{2}\right)^2+\left(u_g+\frac{\rho_g}{2}\right)\left(u_f+\frac{\rho_g}{2}\right)+\left(u_f+\frac{\rho_f}{2}\right)^2\right| \cr
	&\quad + \frac{1}{3}\left(|u_g-u_f|+\frac{|\rho_f-\rho_g|}{2}\right)\left|\left(u_g-\frac{\rho_g}{2}\right)^2+\left(u_g-\frac{\rho_g}{2}\right)\left(u_f-\frac{\rho_f}{2}\right)+\left(u_f-\frac{\rho_f}{2}\right)^2\right| \cr
	&\le C\left(|\rho_f-\rho_g|+|u_f-u_g|\right).
	\end{align*}
	
	\vspace{.2cm}
	
$\bullet$ Estimate of $III$: Since one of the supports of $M[f]$ and $M[g]$ is completely contained within the other, we observe
	\begin{align*}
	III&=\left|\rho_f+\frac13\left\{\left(u_f+\frac{\rho_f}{2}\right)^3-\left(u_f-\frac{\rho_f}{2}\right)^3 \right\}-\rho_g-\frac13\left\{\left(u_g+\frac{\rho_g}{2}\right)^3-\left(u_g-\frac{\rho_g}{2}\right)^3 \right\} \right|.
	\end{align*}
In a similar manner as in the estimate of $II$, we have
	\begin{align*}
	III \le C\left(|\rho_f-\rho_g|+|u_f-u_g|\right).
	\end{align*}
Combining all of the above estimates concludes the desired result for the one-dimensional case. 

\subsection{Proof of Lemma \ref{lipschitz} in the case $d\geq 2$ and $\gamma=\frac{d+2}{d}$}\label{ssec_bdy}
We next provide the proof for the multi-dimensional case. 
	

We decompose the domain $\mathbb{R}^d$ into four parts as 
	\begin{align*}
	\sfD_1&:=\left\{v\in\mathbb{R}^d : |u_f-u_g|> r_f+r_g  \right\},\cr
	\sfD_2&:=\left\{v\in\mathbb{R}^d :  \left|r_f-r_g \right|\le |u_f-u_g|\le r_f+r_g\quad \text{and}\quad |u_f-u_g|^2> \left|r_f^2-r_g^2 \right|  \right\},\cr
	\sfD_3&:=\left\{v\in\mathbb{R}^d : |r_f-r_g|\le |u_f-u_g|\le r_f+r_g\quad \text{and}\quad |u_f-u_g|^2\le\left|r_f^2-r_g^2\right|  \right\}, \quad \mbox{and}\cr
	\sfD_4&:=\left\{v\in\mathbb{R}^d :  |u_f-u_g|\le |r_f-r_g| \right\},
	\end{align*}	
	where we used the following notations: 
	$$
	r_f:=c_d^{\frac 1d}\rho_f^{\frac 1d}, \quad r_g:=c_d^{\frac 1d}\rho_g^{\frac 1d},\quad U_{f,g}:=|u_f-u_g|,\quad \theta_f:=\arccos \left(\frac{r_f^2+U_{f,g}^2-r_g^2}{2r_f U_{f,g}}\right),
	$$
and
	\[
	\tilde{\theta}:=\arccos\left(\frac{r_f^2-r_g^2-U_{f,g}^2}{2r_gU_{f,g}}\right).
	\]
	We now calculate the integral of the weight function $1+|v|^2$ over the symmetric difference of supports for $M[f]$ and $M[g]$ represented by each $\sfD_i$ $(i=1,\dots,4).$ For this, we deal with the weight function $1$ and $|v|^2$ separately and only prove the case $r_g\le r_f$ by symmetry.
	
\vspace{.2cm}

\noindent {\bf ($L^1$-estimate of $M[f] - M[g]$)} We denote
	\[
	\sum_{i=1}^4\int_{\sfD_i} \left|M[f]-M[g]\right| dv =: I + II + III + IV.
	\]
	
	\vspace{.2cm}
	
	$\bullet$ Estimates of $I$ and $IV$: By \eqref{moment_comp}, we get
	\begin{align*}
	I=\int_{\mathbb{R}^d} M[f]\,dv+\int_{\mathbb{R}^d} M[g]\,dv=\rho_f+\rho_g\le \left(\rho_f^{\frac 1d}+\rho_g^{\frac 1d} \right)^d\le C|u_f-u_g|^d
	\end{align*}
	due to the condition of $\sfD_1$. We then use the bound assumption on $u_f$ and $u_g$ to deduce
	\begin{align}\label{I_1.}\begin{split}
	I\le C|u_f-u_g|.
	\end{split}\end{align}
	For the estimate of $IV$, we readily observe
	\begin{equation}\label{I_4.}
	IV=\left|\int_{\mathbb{R}^d}  M[f] \,dv-\int_{ \mathbb{R}^d}  M[g] \,dv\right|=\left|\rho_f-\rho_g\right|.
	\end{equation}
	
	\vspace{.2cm}
	
	$\bullet$ Estimates of $II$: Note that 
	\begin{align}\label{I_20}
	II=\rho_f+\rho_g-2(|\mathbb{V}_f|+|\mathbb{V}_g|),
	\end{align}
	where $\mathbb{V}_f$ and $\mathbb{V}_g$ denote the spherical caps, see Fig. \ref{Spherical cap}.  We translate $u_f$ to the origin to express the region $\mathbb{V}_f$ in spherical coordinates as 
	\begin{equation}\label{cap} 
	\mathbb{V}_f=\left\{ (r,\theta,\varphi_1,\dots,\varphi_{d-2}) : \frac{\cos\theta_f}{\cos \theta}r_f\le r\le r_f,\ 0\le \theta\le \theta_f,\ 0\le \varphi_1,\dots,\varphi_{d-3}\le \pi,\  0\le \varphi_{d-2}< 2\pi \right\},
	\end{equation}
	where $\theta_f$ is strictly less than $\pi/2$. 
 
	 \begin{figure}[!h]
	\begin{center}
		\includegraphics[scale=0.5]{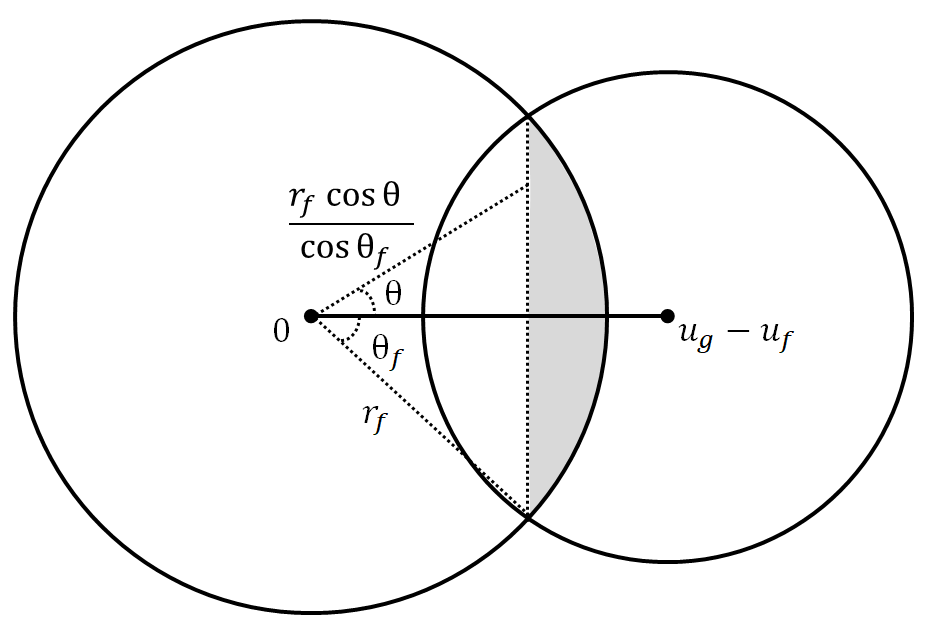}
		\caption{
			Illustration of the spherical cap $\mathbb{V}_f$
		}
 		\label{Spherical cap}
	\end{center}
\end{figure}

	We then have
	\begin{align*}
	|\mathbb{V}_f|&=\int_{0}^{2\pi}\int_{0}^{\pi}\cdots\int_{0}^{\pi}\int_0^{\theta_f}\int_{\frac{\cos\theta_f }{\cos\theta}r_f}^{r_f}r^{d-1}(\sin^{d-2}\theta)( \sin^{d-3}\varphi_{1})\cdots (\sin\varphi_{d-3}) \,drd\theta d\varphi_1\cdots d\varphi_{d-2}\cr
	&=|\mathbb{S}_{d-2}| \int_0^{\theta_f}\int_{\frac{\cos\theta_f }{\cos\theta}r_f}^{r_f}r^{d-1}(\sin^{d-2}\theta)\,drd\theta.
	\end{align*}
	Further calcuation gives
	\begin{align}\label{Vf1}\begin{split}
	|\mathbb{V}_f|&=\frac{|\mathbb{S}_{d-2}|}{d}  r_f^d \int_0^{\theta_f}\left(1-\frac{\cos^d\theta_f}{\cos^d\theta}\right)(\sin^{d-2}\theta)\,d\theta\cr
	&=\frac{|\mathbb{S}_{d-2}|}{|\mathbb{S}_{d-1}|}\rho_f\left(-\frac{1}{d-1}(\sin^{d-1}\theta_f)\cos\theta_f +\int_0^{\theta_f}\sin^{d-2}\theta\,d\theta\right).
	\end{split} \end{align}
	Using the following relations
	$$
	\int \sin^n x\,dx=-\frac 1n(\sin^{n-1}x)\cos x+\frac{n-1}{n}\int \sin^{n-2}x\,dx\quad \mbox{and} \quad |\mathbb{S}_{d+1}| =\frac{2\pi}{d}|\mathbb{S}_{d-1}|,
	$$
	we find
	\begin{align}\label{even}\begin{split}
	&\frac{|\mathbb{S}_{d-2}|}{|\mathbb{S}_{d-1}|}\rho_f   \int_0^{\theta_f}\sin^{d-2}\theta\,d\theta \cr
	&\quad =\frac{|\mathbb{S}_{d-2}|}{|\mathbb{S}_{d-1}|}\rho_f\left(-\sum_{i=1}^{\frac{d-2}{2}}\frac{\prod_{j=1}^{i-1}(d-(2j+1))}{\prod_{j=1}^{i}(d-2j)} (\sin^{d-(2i+1)}\theta_f)\cos \theta_f+\frac{d-3}{d-2 }\frac{d-5}{d-4 }\cdots\frac{1}{2}\int_0^{\theta_f}\,d\theta\right)\cr
	&\quad =-\frac{|\mathbb{S}_{d-2}|}{|\mathbb{S}_{d-1}|}\rho_f\sum_{i=1}^{\frac{d-2}{2}}\frac{\prod_{j=1}^{i-1}(d-(2j+1))}{\prod_{j=1}^{i}(d-2j)} (\sin^{d-(2i+1)}\theta_f)\cos \theta_f+\frac{|\mathbb{S}_{0}|}{|\mathbb{S}_{1}|}\rho_f\theta_f\cr
	&\quad =-\frac{|\mathbb{S}_{d-2}|}{|\mathbb{S}_{d-1}|}\rho_f\sum_{i=1}^{\frac{d-2}{2}}\frac{\prod_{j=1}^{i-1}(d-(2j+1))}{\prod_{j=1}^{i}(d-2j)} (\sin^{d-(2i+1)}\theta_f)\cos \theta_f+\frac{1}{2}\rho_f -\frac{1}{\pi}\rho_f\arcsin\left(\frac{r_f^2+U_{f,g}^2-r_g^2}{2r_f U_{f,g}}\right)
	\end{split}\end{align} 
	for even number $d\ge2$. In the last line, we used
	\begin{equation}\label{arc}
	\theta_f=\arccos \left(\frac{r_f^2+U_{f,g}^2-r_g^2}{2r_f U_{f,g}}\right)=\frac{\pi}{2}-\arcsin \left(\frac{r_f^2+U_{f,g}^2-r_g^2}{2r_f U_{f,g}}\right).
	\end{equation}
	Since $\theta_f$ is strictly less than $\pi/2$, it follows from \eqref{Vf1} and \eqref{even} that
	\begin{align*} 
	\rho_f-2|\mathbb{V}_f|&=2\frac{|\mathbb{S}_{d-2}|}{|\mathbb{S}_{d-1}|}\rho_f\sum_{i=1}^{\frac{d-2}{2}}\frac{\prod_{j=1}^{i-1}(d-(2j+1))}{\prod_{j=1}^{i}(d-2j)} 	(\sin^{d-(2i+1)}\theta_f)\cos \theta_f+\frac{2}{\pi}\rho_f\arcsin\left(\frac{r_f^2+U_{f,g}^2-r_g^2}{2r_f U_{f,g}}\right)\cr
	&\le C_d\left(\rho_f \cos \theta_f+\rho_f\arcsin\left(\frac{r_f^2+U_{f,g}^2-r_g^2}{2r_f U_{f,g}}\right)\right)\cr
	&\le C\rho_f  \frac{r_f^2+U_{f,g}^2-r_g^2}{2r_f U_{f,g}},
	\end{align*}
	where we used the fact that there exists a positive constant $C$ satisfying
	\begin{equation}\label{sin_eq}
	\arcsin x\le Cx,\qquad \forall \, x\in [0,1].
	\end{equation}
	Applying the condition of $\sfD_2$ and \eqref{bounds}, we get
	\begin{align}\label{even2}
	\rho_f-2|\mathbb{V}_f|\le C\rho_f \left(  \frac{r_f^2-r_g^2}{2r_f U_{f,g}} +  \frac{U_{f,g}}{2r_f }\right) \le C\rho_f^{\frac{d-1}{d}} |u_f-u_g| \le C|u_f-u_g|.
	\end{align}
	In the case of odd number $d\ge2$, the same calculation as in \eqref{even} gives
	\begin{align} \begin{split}\label{odd0}
	&\frac{|\mathbb{S}_{d-2}|}{|\mathbb{S}_{d-1}|}\rho_f \int_0^{\theta_f}\sin^{d-2}\theta\,d\theta \cr
	&\quad =\frac{|\mathbb{S}_{d-2}|}{|\mathbb{S}_{d-1}|}\rho_f\left(-\sum_{i=1}^{\frac{d-3}{2}}\frac{\prod_{j=1}^{i-1}(d-(2j+1))}{\prod_{j=1}^{i}(d-2j)} (\sin^{d-	(2i+1)}\theta_f)\cos \theta_f+\frac{d-3}{d-2 }\frac{d-5}{d-4 }\cdots\frac{2}{3}\int_0^{\theta_f}\sin\theta\,d\theta\right)\cr
	&\quad =-\frac{|\mathbb{S}_{d-2}|}{|\mathbb{S}_{d-1}|}\rho_f \sum_{i=1}^{\frac{d-3}{2} }\frac{\prod_{j=1}^{i-1}(d-(2j+1))}{\prod_{j=1}^{i}(d-2j)} (\sin^{d-	(2i+1)}\theta_f)\cos \theta_f +\frac{|\mathbb{S}_{1}|}{|\mathbb{S}_{2}|}\rho_f\left(1-\cos\theta_f\right)\cr
	&\quad =-\frac{|\mathbb{S}_{d-2}|}{|\mathbb{S}_{d-1}|}\rho_f \sum_{i=1}^{\frac{d-3}{2} }\frac{\prod_{j=1}^{i-1}(d-(2j+1))}{\prod_{j=1}^{i}(d-2j)} (\sin^{d-(2i+1)}\theta_f)\cos \theta_f +\frac{1}{2}\rho_f -\frac{1}{2}\rho_f\cos\theta_f,
	\end{split}\end{align} 
	which also leads to
	\begin{align*}\begin{split}
	\rho_f-2|\mathbb{V}_f|&\le C\rho_f \left(  \frac{r_f^2-r_g^2+U_{f,g}^2}{2r_f U_{f,g}}  \right). 
	\end{split}\end{align*}
	The same holds true for $\mathbb{V}_g$, so \eqref{I_20} can be estimated as
	\begin{align}\label{I_2}
	II\le  C|u_f-u_g|.
	\end{align}
	
	\vspace{.2cm}
	
	$\bullet$ Estimate of $III$: In this case, we get
	\begin{equation*}
	III = \rho_f -\rho_g + 2 |\mathsf{supp}(M[g])\setminus \mathsf{supp}(M[f])|=:\rho_f -\rho_g + 2|\mathbb{B}_g\setminus \mathbb{B}_f|,
	\end{equation*}
	and if we translate $u_g$ to the origin, then the region $\mathbb{B}_g\setminus \mathbb{B}_f$ is expressed in spherical coordinates as
	\begin{equation}\label{wI3}
	\mathbb{B}_g\setminus \mathbb{B}_f=\left\{ (r,\theta,\varphi_1,\dots,\varphi_{d-2}) : \tilde{r}\le r\le r_g,\ 0\le \theta\le \tilde{\theta},\ 0\le \varphi_1,\dots,\varphi_{d-3}\le \pi,\  0\le \varphi_{d-2}< 2\pi \right\},
	\end{equation}
	where $\tilde{r}$ denotes
	$$
	\tilde{r}:=\sqrt{r_f^2-U_{f,g}^2(1-\cos^2\theta)}-U_{f,g}\cos\theta
	$$
	(see Fig \ref{Vf-Vg}).
		 \begin{figure}[!h]
		\begin{center}
			\includegraphics[scale=0.3]{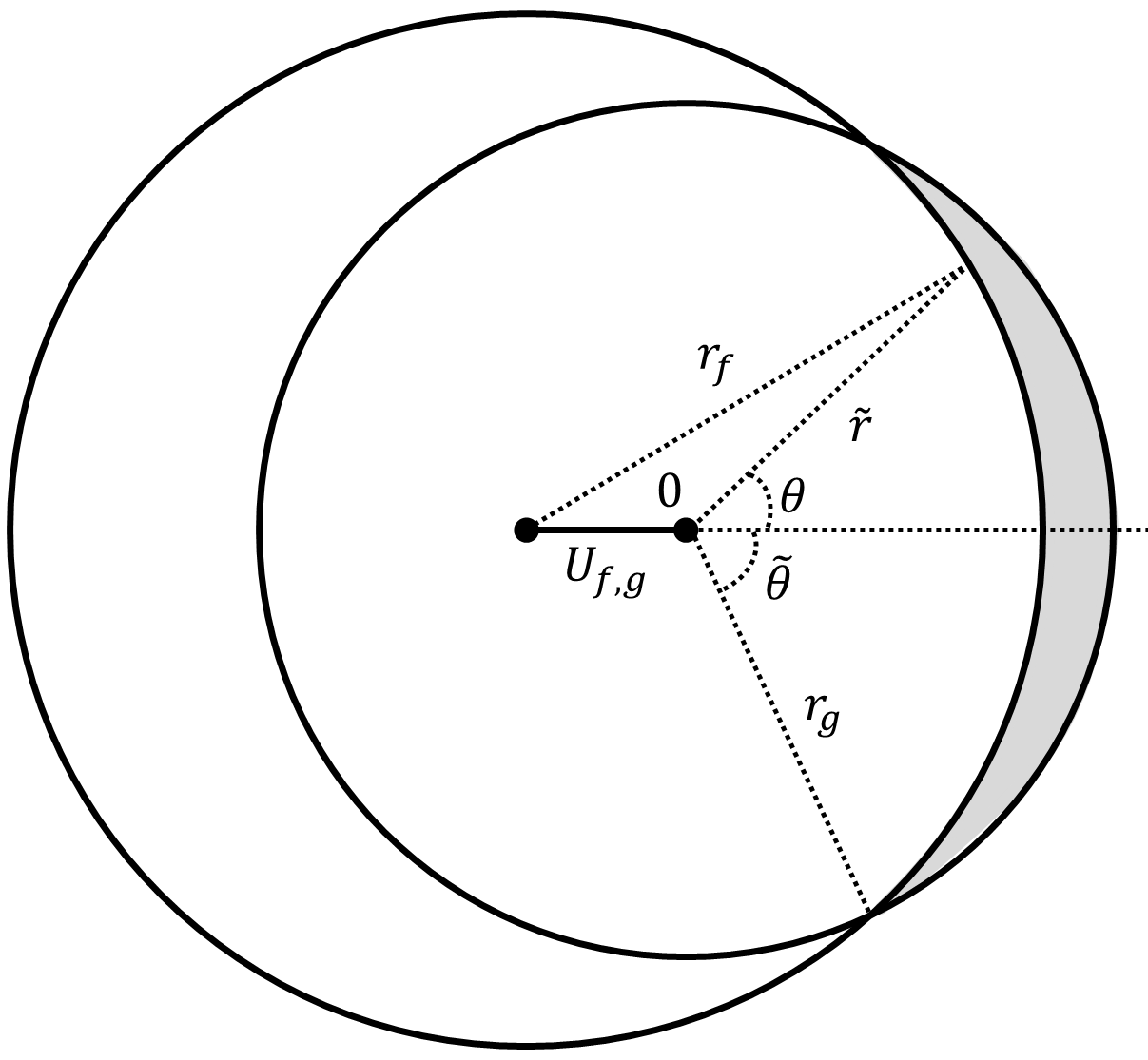}
			\caption{
				Illustration of the region $\mathbb{B}_g\setminus \mathbb{B}_f$
			}
			\label{Vf-Vg}
		\end{center}
	\end{figure} 	
	Since $\tilde{r}$ is increasing in $\theta \in [0,\frac\pi2]$, we see that
	\begin{align}\begin{split}\label{tilde r}
	\tilde{r}\ge r_f-U_{f,g}\ge 0,
	\end{split}\end{align}
	which gives
	\begin{align*}
	|\mathbb{B}_g\setminus \mathbb{B}_f|&=\frac{|\mathbb{S}_{d-2}|}{d}  \int_0^{\tilde{\theta}}\left(r_g^d-\tilde{r}^d\right)\sin^{d-2}\theta\,d\theta\cr
	&\le  C \left(r_g^d-(r_f-U_{f,g})^d\right) \int_0^{\tilde{\theta}}\sin^{d-2}\theta\,d\theta\cr
	&\le  C \left|r_g^d-r_f^d-\sum_{k=1}^{d}(-1)^k \ _dC_k r_f^{d-k} U_{f,g}^k\right|.
	\end{align*}
	Then, it follows from \eqref{bounds}  that
	\begin{align*}
	|\mathbb{B}_g\setminus \mathbb{B}_f|&\le  C \left|r_g^d-r_f^d\right|+C\sum_{k=1}^{d} r_f^{d-k} U_{f,g}^k\cr
	&\le  C \left(\left|\rho_f-\rho_g\right| +|u_f-u_g|\right),
	\end{align*}
	which implies
	\begin{align*}
	III\le C\left(|\rho_f-\rho_g|+|u_f-u_g|\right).
	\end{align*}
	This, combined with \eqref{I_1.}, \eqref{I_4.}  and \eqref{I_2}, gives the desired result.

\vspace{.2cm}

\noindent {\bf ($L^1$-estimate of $|v|^2(M[f] - M[g])$)} Here, we denote 
	\[
	\sum_{i=1}^4\int_{\sfD_i} |v|^2\left|M[f]-M[g]\right| dv =: \widetilde I + \widetilde{II} + \widetilde{III} + \widetilde{IV}.
	\]
	
	\vspace{.2cm}
	
	$\bullet$ Estimates of $\widetilde I$ and $\widetilde{IV}$: It follows from \eqref{moment_comp} that
	\begin{align*}\begin{split}
	\widetilde I &=	\rho_f |u_f|^2+\rho_g |u_g|^2+\frac{|\mathbb{S}_{d-1}|}{d+2}\left(\frac{d}{|\mathbb{S}_{d-1}|}\right)^{\frac{d+2}{d}} \left( \rho_f^{\frac{d+2}{d}} +\rho_g^{\frac{d+2}{d}}\right)\cr
	&\le ( \rho_f+\rho_g) |u_f|^2+\rho_g  |u_f+u_g||u_f-u_g|+C_d\left( \rho_f^{\frac{d+2}{d}} +\rho_g^{\frac{d+2}{d}}\right)\cr
	&\le C\left( |u_f|^2|u_f-u_g|^d+\rho_g  |u_f+u_g||u_f-u_g|+|u_f-u_g|^{d+2} \right),
	\end{split}	\end{align*}
	where we used the condition of $\sfD_1$:
	$$
	\rho_f^{\frac 1d}+\rho_g^{\frac 1d} \le C|u_f-u_g|.
	$$
	We also estimate 
	\begin{align*}\begin{split}
	 \widetilde{IV} &=\left|\rho_f |u_f|^2-\rho_g |u_g|^2+\frac{d}{d+2}|V_{d}|\left(\frac{d}{|S_{d-1}|}\right)^{\frac{d+2}{d}}\left( \rho_f^{\frac{d+2}{d}} -\rho_g^{\frac{d+2}{d}}\right)\right|\cr
	&\le  \left|\rho_f-\rho_g \right| |u_f|^2+\rho_g |u_f+u_g||u_f-u_g|+C (\rho_f+\rho_g)^{\frac 2d}\left|\rho_f-\rho_g\right|.
	\end{split}\end{align*}
	Using \eqref{bounds}, we obtain
	\begin{equation*}
	\widetilde I + \widetilde{IV} \le C\left(|\rho_f-\rho_g|+|u_f-u_g|\right).
	\end{equation*}
	
		\vspace{.2cm}
	
	$\bullet$ Estimates of $\widetilde{II}$: A straightforward computation gives
	\begin{align}\label{I_2 last}\begin{split}
	\widetilde{II} &=\int_{\mathbb{R}^d}|v|^2 M[f] \,dv+	\int_{\mathbb{R}^d}|v|^2  M[g] \,dv-2\int_{\mathbb{V}_f\cup \mathbb{V}_g}|v|^2  \,dv\cr
	&=	\rho_f |u_f|^2+\rho_g |u_g|^2+\frac{|\mathbb{S}_{d-1}|}{d+2}\left(\frac{d}{|\mathbb{S}_{d-1}|}\right)^{\frac{d+2}{d}}\left( \rho_f^{\frac{d+2}{d}} +\rho_g^{\frac{d+2}{d}}\right)-2\int_{\mathbb{V}_f\cup \mathbb{V}_g}|v|^2  \,dv.
	\end{split}\end{align}
	 If we translate $u_f$ to the origin, we get
	\begin{align*}\begin{split}
	\int_{\mathbb{V}_f}|v+u_f|^2\,dv&=|u_f|^2|\mathbb{V}_f|+2 \int_{\mathbb{V}_f}u_f\cdot v\,dv+\int_{\mathbb{V}_f}|v|^2\,dv,
	\end{split}\end{align*}
where $\mathbb{V}_f$ indicates the region \eqref{cap}. In this way, the last term of \eqref{I_2 last} can be decomposed into
	\begin{align}\label{V_f}\begin{split}
	-2\int_{\mathbb{V}_f\cup \mathbb{V}_g}|v|^2  \,dv&=-2\left(|u_f|^2|\mathbb{V}_f|+|u_g|^2|\mathbb{V}_g|\right)-4 \left(\int_{\mathbb{V}_f}u_f\cdot v\,dv+\int_{\mathbb{V}_g}u_g\cdot v\,dv\right)\cr
	&\quad -2\left(\int_{\mathbb{V}_f}|v|^2\,dv+\int_{\mathbb{V}_g}|v|^2\,dv\right)\cr
	&=: \widetilde{II}_1 + \widetilde{II}_2 + \widetilde{II}_3,
	\end{split}\end{align}
	where we use a similar argument as in \eqref{even2} to estimate
	\begin{align}\label{21}
	 \widetilde{II}_1&\le -( \rho_f|u_f|^2+\rho_f|u_f|^2) +C|u_f-u_g|.
	\end{align}
For the estimate of $\widetilde{II}_2$, we need to calculate $u_f\cdot v$ in spherical coordinates. Consider the rotation matrix:
	\begin{align*}
	R:=R_{v_{1}}(\theta_{1})R_{v_{2}}(\theta_{2})R_{v_{3}}(\theta_{3})\cdots R_{v_{d-1}}(\theta_{d-1}),
	\end{align*}
	where each  $R_{v_{i}}$ represents the counterclockwise rotation of the $v_i$-axis in direction to the $v_d$-axis:
	\begin{align*}
	R_{v_{i}}(\theta_{i})=\begin{bmatrix}
	a_{jk}
	\end{bmatrix}_{d\times d}
	\end{align*}
	with
	$$
	a_{ii}=\cos\theta_{i}, \qquad  a_{dd}=\cos\theta_{i}, \qquad a_{id}=-\sin\theta_{i}, \qquad a_{di}=\sin\theta_{i}, 
	$$
	$$
	\ a_{jj}=1\quad \text{if}\quad j\neq i,d, \qquad \ a_{jk}=0\quad \text{otherwise},
	$$
	and $\theta_{i}$ is the angle  between the vector $u_g-u_f$ and $v_d$-axis, which is described in Fig \ref{angle}. This gives
	$$
	\cos\theta_{i}=\frac{ u_{g_d}-u_{f_d} }{|u_g-u_f|}\quad \mbox{and} \quad \sin\theta_{i}=\frac{ u_{g_i}-u_{f_i} }{|u_g-u_f|}.
	$$
	Here $u_{g_k}-u_{f_k}$ denotes the $k$-th component of $u_g-u_f\in \mathbb{R}^d$, i.e. the matrix $R$ stands for the rotation of $u_g-u_f$ to the $v^d$-axis.
\begin{figure}[!h]
		\begin{center}
			\includegraphics[scale=0.6]{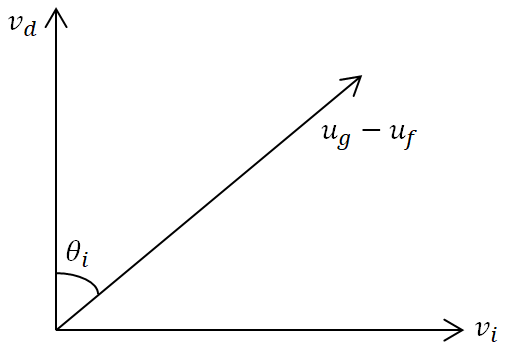}
			\caption{
Illustration of the angle $\theta_i$.
}
\label{angle}
		\end{center}
	\end{figure} 
	
 Note that the $d\times d$ matrix $R$ takes the form of
	\begin{equation}\label{rotation}
	\begin{bmatrix}
	*&*&*&\cdots &*&*\cr
		\vdots&\vdots&\vdots&\ddots &\vdots&\vdots\cr
		*&*&*&\cdots &*&*\\[3mm]		
	\sin \theta_{1} & \cos \theta_{1}\sin \theta_{2} & \cos \theta_{1}\cos \theta_{2}\sin \theta_{3}& \cdots & \displaystyle \left(\prod_{i=1}^{d-2}\cos \theta_{i}\right)\sin\theta_{d-1} & \displaystyle \prod_{i=1}^{d-1}\cos \theta_{i}
	\end{bmatrix}.
	\end{equation}
	Next, we apply the change of variables $p:=Rv$ to the region $\mathbb{V}_f$ in order to measure the angle $\theta$ from the $v_d$-axis. Since the inner product is invariant under rotations, we then have
	\begin{align*} \begin{split}
	\int_{\mathbb{V}_f}u_f\cdot v\,dv&=\int_{\mathbb{V}_f}u_f\cdot (R^{-1}p)\left|\det R^{-1}\right|\,dp=(Ru_{f})\cdot\int_{\mathbb{V}_f} p \,dp.
	\end{split}\end{align*} 
	Applying the spherical coordinates \eqref{cap}:
	$$
	p=r\left(\prod_{i=1}^{d-2}(\sin\varphi_{i})\sin\theta,\prod_{i=1}^{d-3}(\sin\varphi_{i})(\cos\varphi_{d-2})\sin\theta, \dots, (\sin\varphi_1)(\cos\varphi_2) \sin\theta, (\cos\varphi_1) \sin\theta, \cos \theta\right)
	$$
	and using the following identities:
	$$
	\int_0^{2\pi}\sin\varphi_{d-2}\,d\varphi_{d-2}=0,\quad \int_0^{2\pi}\cos\varphi_{d-2}\,d\varphi_{d-2}=0,\quad \mbox{and} \quad \int_0^\pi (\sin^{d-2-k}\varphi_k) \cos\varphi_k\,d\varphi_k=0
	$$
	for all $k=1,\dots,d-3$, we deduce
	\begin{align*} \begin{split}
	&\int_{\mathbb{V}_f}u_f\cdot v\,dv\cr
	&\quad =(Ru_f)\cdot\int_{0}^{2\pi}\int_{0}^{\pi}\cdots\int_{0}^{\pi}\int_0^{\theta_f}\int_{\frac{\cos\theta_f }{\cos\theta}r_f}^{r_f}p\ r^{d-1}(\sin^{d-2}\theta)( \sin^{d-3}\varphi_{1})\cdots (\sin\varphi_{d-3}) \,drd\theta d\varphi_1\cdots d\varphi_{d-2}\cr
	&\quad =(Ru_f)\cdot\left(0,0,\cdots,0,1   \right)|\mathbb{S}_{d-2}|\int_0^{\theta_f}\int_{\frac{\cos\theta_f }{\cos\theta}r_f}^{r_f} r^{d}\cos\theta(\sin^{d-2}\theta)  \,drd\theta.
	\end{split}\end{align*} 
	This gives
	\begin{align}\label{wbad1}\begin{split}
	&\int_{\mathbb{V}_f}u_f\cdot v\,dv\cr
	&\quad =|\mathbb{S}_{d-2}|\left(-\sum_{i=1}^{d-1}u_{f_i}\sin\theta_i\prod_{j=1}^{i-1}\cos\theta_j+u_{f_d}\prod_{j=1}^{d-1}\cos\theta_j\right)\int_0^{\theta_f}\int_{\frac{\cos\theta_f }{\cos\theta}r_f}^{r_f} r^{d}\cos\theta(\sin^{d-2}\theta)  \,drd\theta
	\end{split}\end{align} 
	thanks to \eqref{rotation}.
	By a direct calculation as in \eqref{Vf1}, we get
	\begin{align*}
	\int_0^{\theta_f}\int_{\frac{\cos\theta_f }{\cos\theta}r_f}^{r_f} r^{d}\cos\theta\sin^{d-2}\theta  \,drd\theta&=\frac{r_f^{d+1}}{d+1}\int_0^{\theta_f} \left\{1-\left(\frac{\cos\theta_f}{\cos\theta}\right)^{d+1}\right\}\cos\theta(\sin^{d-2}\theta)  \,d\theta\cr
	&=\frac{r_f^{d+1}}{d+1}\left\{\int_0^{\theta_f}\cos\theta(\sin^{d-2}\theta)  \,d\theta-\frac{1}{d-1}(\sin^{d-1}\theta_f)\cos^2\theta_f\right\} \cr
	&=\frac{r_f^{d+1}}{d+1}\frac{\sin^{d-1}\theta_f}{d-1}\left\{1-\cos^2\theta_f\right\} .
	\end{align*} 
	This together with \eqref{wbad1} yields
	\begin{align}\label{wbad2}\begin{split}
	\int_{\mathbb{V}_f}u_f\cdot v\,dv&=\frac{|\mathbb{S}_{d-2}|}{(d-1)(d+1)}r_f^{d+1}(\sin^{d-1}\theta_f)\left(-\sum_{i=1}^{d-1}u_{f_i}\sin\theta_i\prod_{j=1}^{i-1}\cos\theta_j+u_{f_d}\prod_{j=1}^{d-1}\cos\theta_j\right)\cr
	&-\frac{|\mathbb{S}_{d-2}|}{(d-1)(d+1)}r_f^{d+1}(\sin^{d-1}\theta_f)\cos^2\theta_f\left(-\sum_{i=1}^{d-1}u_{f_i}\sin\theta_i\prod_{j=1}^{i-1}\cos\theta_j+u_{f_d}\prod_{j=1}^{d-1}\cos\theta_j\right).
	\end{split}\end{align} 
	In the case of $\mathbb{V}_g$, we translate $u_g$ to the origin, so the region $\mathbb{V}_g$ is positioned in the opposite direction to the region $\mathbb{V}_f$ with respect to the origin. Applying the rotation \eqref{rotation} to $\mathbb{V}_g$, the angle $\theta$ ranges from $\pi-\theta_g$ to $\pi$. Thus we have 
	\begin{align*}
	&\int_{\mathbb{V}_g}u_g\cdot v\,dv\cr
	&\quad =|\mathbb{S}_{d-2}|\left(-\sum_{i=1}^{d-1}u_{g_i}\sin\theta_i\prod_{j=1}^{i-1}\cos\theta_j+u_{g_d}\prod_{j=1}^{d-1}\cos\theta_j\right)\int_{\pi-\theta_g}^{\pi}\int_{\frac{\cos\theta_g }{\cos\theta}r_g}^{r_g} r^{d}\cos\theta(\sin^{d-2}\theta)  \,drd\theta\cr
	&\quad =|\mathbb{S}_{d-2}|\left(-\sum_{i=1}^{d-1}u_{g_i}\sin\theta_i\prod_{j=1}^{i-1}\cos\theta_j+u_{g_d}\prod_{j=1}^{d-1}\cos\theta_j\right)\frac{r_g^{d+1}}{d+1}\int_{\pi-\theta_g}^{\pi} \left\{1-\left(\frac{\cos\theta_g}{\cos\theta}\right)^{d+1}\right\}\cos\theta(\sin^{d-2}\theta)  \,d\theta\cr
	&\quad =-\frac{|\mathbb{S}_{d-2}|}{(d-1)(d+1)}r_g^{d+1}(\sin^{d-1}\theta_g)\left(-\sum_{i=1}^{d-1}u_{g_i}\sin\theta_i\prod_{j=1}^{i-1}\cos\theta_j+u_{g_d}\prod_{j=1}^{d-1}\cos\theta_j\right)\cr
	&\qquad+(-1)^{d+1}\frac{|\mathbb{S}_{d-2}|}{(d-1)(d+1)}r_g^{d+1}(\sin^{d-1}\theta_g)\cos^2\theta_g\left(-\sum_{i=1}^{d-1}u_{g_i}\sin\theta_i\prod_{j=1}^{i-1}\cos\theta_j+u_{g_d}\prod_{j=1}^{d-1}\cos\theta_j\right),
	\end{align*}
	which, combined with \eqref{wbad2} gives
	\begin{align*}\begin{split}
	\widetilde{II}_2 &=-4 \int_{\mathbb{V}_f}u_f\cdot v\,dv-4\int_{\mathbb{V}_g}u_g\cdot v\,dv\cr
	&=-\frac{|\mathbb{S}_{d-2}|}{(d-1)(d+1)}\biggl\{r_f^{d+1}\sin^{d-1}\theta_f\biggl(-\sum_{i=1}^{d-1}u_{f_i}\sin\theta_i\prod_{j=1}^{i-1}\cos\theta_j+u_{f_d}\prod_{j=1}^{d-1}\cos\theta_j\biggl)\cr
	&\quad -r_g^{d+1}\sin^{d-1}\theta_g\biggl(-\sum_{i=1}^{d-1}u_{g_i}\sin\theta_i\prod_{j=1}^{i-1}\cos\theta_j+u_{g_d}\prod_{j=1}^{d-1}\cos\theta_j\biggl)\biggl\}\cr
	&\quad +\frac{|\mathbb{S}_{d-2}|}{(d-1)(d+1)}\biggl\{r_f^{d+1}\sin^{d-1}\theta_f\cos^2\theta_f\biggl(-\sum_{i=1}^{d-1}u_{f_i}\sin\theta_i\prod_{j=1}^{i-1}\cos\theta_j+u_{f_d}\prod_{j=1}^{d-1}\cos\theta_j\biggl)\cr
	&\quad -(-1)^{d+1}r_g^{d+1}\sin^{d-1}\theta_g\cos^2\theta_g\biggl(-\sum_{i=1}^{d-1}u_{g_i}\sin\theta_i\prod_{j=1}^{i-1}\cos\theta_j+u_{g_d}\prod_{j=1}^{d-1}\cos\theta_j\biggl)\biggl\}\cr
	&=: \widetilde{II}_{21} + \widetilde{II}_{22}.
	\end{split}\end{align*}
	To estimate $\widetilde{II}_{21}$, we observe from \eqref{bounds}  that
	$$
	\left|r_f^{d+1}-r_g^{d+1}\right|+  \left|u_{f_k}-u_{g_k}\right|\le C\left(|\rho_f-\rho_g|+|u_f-u_g|\right).
	$$
	Moreover, it follows from \eqref{arc} and the condition of $\sfD_2$ that
	\begin{align*}
	\left|(\sin^{d-1}\theta_f)-(\sin^{d-1}\theta_g)\right|&\le C\left|\theta_f-\theta_g\right| \le C|u_f-u_g|,
	\end{align*}
	due to \eqref{sin_eq}.	From those observations, we see that applying the triangle inequality to $\widetilde{II}_{21}$ provides
	\begin{align*}
	\widetilde{II}_{21} \le C\left(|\rho_f-\rho_g|+|u_f-u_g|\right).
	\end{align*}
	Similarly, $\widetilde{II}_{22}$ can be estimated as
	\begin{align*}
\widetilde{II}_{22}&\le C \left(\cos\theta_f+\cos\theta_g \right)\le C\left(|\rho_f-\rho_g|+|u_f-u_g|\right),
	\end{align*}
	and hence we conclude that
	\begin{equation}\label{22}
	 \widetilde{II}_{2} \le C\left(|\rho_f-\rho_g|+|u_f-u_g|\right).
	\end{equation}

We now estimate $\widetilde{II}_{23}$. By using similar argument as in \eqref{Vf1}, we find
	\begin{align*}\begin{split}
	&\int_{\mathbb{V}_f}|v|^2\,dv\cr
	&\quad =\int_{0}^{2\pi}\int_{0}^{\pi}\cdots\int_{0}^{\pi}\int_0^{\theta_f}\int_{\frac{\cos\theta_f }{\cos\theta}r_f}^{r_f}r^{d+1}(\sin^{d-2}\theta) (\sin^{d-3}\varphi_{1})\cdots (\sin\varphi_{d-3}) \,drd\theta d\varphi_1\cdots d\varphi_{d-2}\cr
	&\quad =\frac{|\mathbb{S}_{d-2}|}{d+2} r_f^{d+2} \int_0^{\theta_f}\left(1-\frac{\cos^{d+2}\theta_f}{\cos^{d+2}\theta}\right)\sin^{d-2}\theta\,d\theta\cr
	&\quad = \frac{|\mathbb{S}_{d-2}|}{d+2} \left(\frac{d}{|\mathbb{S}_{d-1}|}\right)^{d+2}\rho_f^{\frac{d+2}{d}}\left(-\frac{1}{d+1}(\sin^{d+1}\theta_f)\cos\theta_f-\frac{1}{d-1}(\sin^{d-1}\theta_f)\cos^3\theta_f+ \int_0^{\theta_f}\sin^{d-2}\theta\,d\theta\right).
	\end{split} \end{align*}
	Recall from \eqref{even} and \eqref{odd0} that the last term is calculated as
	\begin{align*}
	&\frac{|\mathbb{S}_{d-2}|}{|\mathbb{S}_{d-1}|}\int_0^{\theta_f}\sin^{d-2}\theta\,d\theta\cr
	&\quad =\frac{1}{2}-\frac{|\mathbb{S}_{d-2}|}{|\mathbb{S}_{d-1}|}\sum_{i=1}^{\frac{d-2}{2}}\frac{\prod_{j=1}^{i-1}(d-(2j+1))}{\prod_{j=1}^{i}(d-2j)} (\sin^{d-(2i+1)}\theta_f)\cos \theta_f -\frac{1}{\pi}\arcsin\left(\frac{r_f^2+U_{f,g}^2-r_g^2}{2r_f U_{f,g}}\right)
	\end{align*} 
	for even number $d\ge 2$, and
	\begin{align*}
	\frac{|\mathbb{S}_{d-2}|}{|\mathbb{S}_{d-1}|}\int_0^{\theta_f}\sin^{d-2}\theta\,d\theta=\frac{1}{2}-\frac{|\mathbb{S}_{d-2}|}{|\mathbb{S}_{d-1}|} \sum_{i=1}^{\frac{d-3}{2} }\frac{\prod_{j=1}^{i-1}(d-(2j+1))}{\prod_{j=1}^{i}(d-2j)} (\sin^{d-(2i+1)}\theta_f) \cos \theta_f  -\frac{1}{2}\cos\theta_f
	\end{align*} 
	for odd number $d\ge 2$.  Since in the case of $\sfD_2$, we proved that
	$$
	 \left|\cos\theta_f\right| +\left|\arcsin\left(\frac{r_f^2+U_{f,g}^2-r_g^2}{2r_f U_{f,g}}\right)\right| \le C\left(|\rho_f-\rho_g|+|u_f-u_g|\right),
	$$
	we obtain
	\begin{align}\label{23}\begin{split}
	 \widetilde{II}_{23} &= -2\int_{\mathbb{V}_f}|v|^2\,dv-2 \int_{\mathbb{V}_g}|v|^2\,dv\cr
	&\le - \frac{|\mathbb{S}_{d-1}|}{d+2} \left(\frac{d}{|\mathbb{S}_{d-1}|}\right)^{d+2}\left(\rho_f^{\frac{d+2}{d}}+\rho_g^{\frac{d+2}{d}}\right)+C\left(|\rho_f-\rho_g|+|u_f-u_g|\right).
	\end{split}\end{align} 
	Finally, we go back to \eqref{I_2 last} with \eqref{V_f}, \eqref{21}, \eqref{22}, and \eqref{23} to conclude that
	\begin{align*}
	\widetilde{II} \le C\left(|\rho_f-\rho_g|+|u_f-u_g|\right).
	\end{align*}
	Now it only remains to estimate $ \widetilde{III}$:
	\begin{align*}
	\widetilde{III}=\intr|v|^2 M[f]\,dv-\intr|v|^2 M[g]\,dv +2\int_{\mathbb{B}_g\setminus \mathbb{B}_f}|v|^2\,dv.
	\end{align*}
	By \eqref{moment_comp}, we get
	\begin{align*}
	\intr|v|^2 M[f]\,dv-\intr|v|^2 M[g]\,dv&\le\rho_f|u_f|^2-\rho_g|u_g|^2+C  \left|\rho_f^{\frac{d+2}{d}}-\rho_g^{\frac{d+2}{d}}\right|\cr
	&\le C\left(|\rho_f-\rho_g|+|u_f-u_g|\right).
	\end{align*}
	Also, it follows from \eqref{bounds}, \eqref{wI3} and \eqref{tilde r} that
	\begin{align*}
	2\int_{\mathbb{B}_g\setminus \mathbb{B}_f}|v|^2\,dv&=2\frac{|\mathbb{S}_{d-2}|}{d+2}  \int_0^{\tilde{\theta}}\left(r_g^{d+2}-\tilde{r}^{d+2}\right)\sin^{d-2}\theta\,d\theta\cr
	&\le  C \left(r_g^{d+2}-\left(r_f-U_{f,g}\right)^{d+2}\right) \int_0^{\tilde{\theta}}\sin^{d-2}\theta\,d\theta\cr
	&\le   C \left|r_g^d-r_f^d-\sum_{k=1}^{d+2}(-1)^k \ _dC_k r_f^{d+2-k} U_{f,g}^k\right| . 
	\end{align*}
	Thus, applying \eqref{bounds}  gives
	\begin{align*}
	\widetilde{III}&\le C\left(|\rho_f-\rho_g|+|u_f-u_g|\right).
	\end{align*}
	This concludes the desired result.
	

%
%
%
%

\section{Proof of Theorem \ref{main_thm}: $\gamma = \frac{d+2}d$ case}\label{sec_bcase}

In this section, we provide the details of the proof for Theorem \ref{main_thm} in the case $\gamma = \frac{d+2}d$. We only deal with the whole space case $\om =\R^d$ since the case of periodic spatial domain can be handled in almost the same manner.

\subsection{Regularized equation}\label{sec_reg}

For the existence of weak solutions to the equation \eqref{BGK}, we first consider the following regularized equation:
\bq\label{reg_eqn}
\pa_t f_\e + v \cdot \nabla_x f_\e = M^\e[f_\e] - f_\e,
\eq
subject to regularized initial data
\[
f_\e(x,v,t)|_{t=0} =: f_{\e,0}(x,v), \quad (x,v)\in \R^d \times \R^d,
\]
where
\bq\label{eq_bcase}
M^\e[f_\e]:= \mathbf{1}_{|u^\e_{f_\e}-v|^d\le c_d\rho^\e_{f_\e}}
\eq
with
\bq\label{reg_fcts}
\rho^\e_{f_\e} := \frac{\rho_{f_\e} * \theta^\e}{1 + \e^{d+1} \rho_{f_\e} * \theta^\e} \quad \mbox{and} \quad u^\e_{f_\e} := \frac{(\rho_{f_\e} u_{f_\e})*\theta^\e}{\rho_{f_\e} * \theta^\e + \e^{2d+1}(1 + |(\rho_{f_\e} u_{f_\e})*\theta^\e|^2)}.
\eq
Here  $\theta^\e:=\e^{-d}\theta(x/\e)$, where $ \theta \in \mc^\infty_c(\om)$ is the standard mollifier, so $\left\|\theta^\e\right\|_{L^1}= 1$. The regularized initial data $f_{\e,0}$ is defined by 
\bq\label{reg_ini}
f_{\e,0} =f_0*\varphi^\e+\e\frac{e^{-|v|^2}}{1+|x|^q},
\eq
where $\varphi^\e:=\theta^\e(x)\theta^\e(y)$ and the constant $q$ is strictly bigger than $d$.
 Throughout this paper, we assume the regularization parameter $\e \leq 1$. Apparently, we formally see that $M^\e[f]$ converges to $M[f]$ as $\e \to 0$. For the existence of weak solutions, we shall need to establish the existence theory for the regularized equation \eqref{reg_eqn}. 
\begin{proposition}\label{prop_reg} Let $T>0$ and the initial data $f_0$ satisfy the assumptions of Theorem \ref{main_thm}. Then there exists a weak solution $f_\e$ of the equations \eqref{reg_eqn}--\eqref{eq_bcase} with \eqref{reg_ini} in the sense of Definition \ref{def_weak}. Moreover, we have
	\[
	\sup_{0 \leq t \leq T}\|f_\e(\cdot,\cdot,t)\|_{L^1\cap L^\infty} \leq \|f_{\e,0}\|_{L^1 \cap L^\infty} +1, \qquad \rho_{f_\e}(x,t) \geq c_\e(1+|x|^q)^{-1},
	\]
	for some $c_\e > 0$, 	and
	\bq\label{kin_en}
	\intrr |v|^2 f_\e(x,v,t)\,dxdv + \int_0^t \intrr |v|^2\lt(f_\e - M[f_\e]\rt)\,dxdvds\leq \intrr |v|^2 f_{\e,0}(x,v)\,dxdv 
	\eq
	for $t \in [0,T]$. In particular, \eqref{kin_en} implies
	\[
\sup_{0\le t\le T}	\intrr |v|^2 f_\e(x,v,t)\,dxdv \leq \intrr |v|^2 f_{\e,0}(x,v)\,dxdv
	\]
	due to the minimization principle \eqref{minimization}.
\end{proposition}

\subsubsection{Regularized and linearized equation}

We construct the solution $f_\e$ to the regularized equation \eqref{reg_eqn} by considering the approximation sequence $f^k_\e$ given as solutions of the following equation:
\bq\label{app_weak_eq}
\pa_t f_\e^{k+1} + v\cdot\nabla_x f_\e^{k+1}   = M^\e[f_\e^k] - f_\e^{k+1},
\eq
with the initial data and first iteration step:
\[
f_\e^k(x,v,t)|_{t=0} = f_{\e,0}(x,v) \quad \mbox{for all } k \geq 1 \quad \mbox{and} \quad f_\e^0(x,v,t) = f_{\e,0}(x,v) \quad \mbox{for } (x,v,t) \in \R^d \times \R^d \times (0,T).
\]
The approximated equilibrium function $M^\e[f_\e^k]$ is given by
\[
M^\e[f_\e^k]:= \mathbf{1}_{|u^\e_{f^k_\e}-v|^d\le c_d\rho^\e_{f^k_\e}}.
\]
In the following, for the sake of notational simplicity, we omit $\e$-dependence in $f_\e^k$, i.e., $f^k = f_\e^k$. In order to study the convergence of approximations $f^k$, we introduce a weighted $L^p$-norm with $p\in[1,\infty)$:
\[
\left\|f\right\|_{L^p_q}:= \intrr \left(1+|v|^q\right) |f(x,v)|^p\,dxdv
\]
and for $p=\infty$, 
\[
\left\|f\right\|_{L^\infty_q}:= \esssup_{(x,v)\in\R^d\times\R^d} \left(1+|v|^q\right) |f(x,v)|.
\]
Naturally, $L^p_q(\R^d \times \R^d)$ with $p \in [1,\infty]$ denotes the space of functions with finite norms. 

For the regularized and linearized equation \eqref{app_weak_eq}, we show the global existence and uniqueness of solutions and uniform-in-$k$ bound estimates. 
\begin{lemma}\label{lem_uni} Let $T>0$ and $q > d+1$. Suppose that the initial data $f_0$ satisfies the assumptions of Theorem \ref{main_thm}. For any $k \in \N$, there exists a unique solution $f^k$ of the equation \eqref{app_weak_eq} such that $f^k \in L^\infty(0,T; L^\infty_q(\R^d \times \R^d))$. Moreover, we have
	\[
	\sup_{k \in \N}\sup_{0 \leq t \leq T}\|f^k(\cdot,\cdot,t)\|_{L^1 \cap L^\infty} \leq \|f_{\e,0}\|_{L^1\cap L^\infty}+1 , \quad \inf_{k \in \N}\rho_{f^k}(x,t) \geq c_\e(1+|x|^q)^{-1},
	\]
	and
	\[
	\sup_{k \in \N}\sup_{0 \leq t \leq T}\|f^k(\cdot,\cdot,t)\|_{L^\infty_q} \leq \|f_{\e,0}\|_{L^\infty_q} + C_\e
	\]
	for some $c_\e > 0$ and $C_\e > 0$ independent of $k$.
\end{lemma}
\begin{proof} We first readily check the existence and uniqueness of solutions $f^k \in L^\infty(0,T; L^\infty_q(\R^d \times \R^d))$ to \eqref{app_weak_eq} by the standard existence theory for transport equations. Thus, we only provide the bound estimates on $\|f^k\|_{L^1\cap L^\infty}$, $\rho_{f^k}$, and $\|f^k\|_{L^\infty_q}$ for $q > d+1$.

\vspace{.2cm}

	$\bullet$ Estimate of $\|f^{k+1}\|_{L^1\cap L^\infty}$: We begin with the estimate of $\|f^{k+1}\|_{L^1}$. Since
	\[
	\intrr M^\e[f^k]\,dxdv = \intr \rho^\e_{f^k}\,dx \leq \intr \rho_{f^k} * \theta_\e\,dx \leq \intr \rho_{f^k}\,dx,
	\]
	we obtain
	\[
	\frac{d}{dt}\intrr f^{k+1}\,dxdv \leq \intrr f^k\,dxdv - \intrr f^{k+1}\,dxdv.
	\]
	This implies $\|f^k(t)\|_{L^1} \leq \|f_{\e,0}\|_{L^1}$ for all $k \geq 0$ and $t \in [0,T]$. 
	
	For the estimate of $L^\infty$-norm of $f^k$, we introduce the following backward characteristics: 
	\begin{align*}
	\frac{d}{ds} X(s) &= V(s),\qquad	\frac{d}{ds} V(s) = 0,
	\end{align*}
	with the terminal data: 
	\[
(X(t),V(t)) = (x,v).
	\]
	Along that characteristics, we have
	\bq\label{mild_f}
	f^{k+1}(x,v,t) = f_{\e,0}(x-vt,v)e^{-t} + \int_0^t e^{-(t-s)} M^\e[f^k](x - v(t-s),v,s)\,ds.
	\eq
	Since $0 \leq  M^\e[f^k] \leq 1$ and $f_0 \geq 0$, we get $f^{k+1} \geq 0$ and 
	\[
	f^{k+1}(x,v,t) \leq  f_{\e,0}(x-vt,v)e^{-t} + 1 - e^{-t}. 
	\]
	Thus $\|f^{k+1}(t)\|_{L^\infty} \leq  \|f_{\e,0}\|_{L^\infty} +1$. Combining this with the $L^1$-estimate gives the desired result.
	
	\vspace{.2cm}
	
	$\bullet$ Estimate of the lower bound on $\rho_{f^{k+1}}$: Since $M^\e \geq 0$, it follows from \eqref{mild_f} that
		\[
	\rho_{f^{k+1}} \geq e^{-T}\intr f_{\e,0}(x-vt,v)\,dv\ge e^{-T}\int_{\R^d} \e\frac{e^{-|v|^2}}{1+|x-vt|^q} \,dv\ge C\e\int_{\R^d} \frac{e^{-|v|^2}}{(1+|x|^q)(1+|v|^q)} \,dv, 
\]
which gives the desired result.
\vspace{.2cm}

	$\bullet$ Estimate of $\|f^{k+1}\|_{L^\infty_q}$: We use
	\bq\label{bdd_e}
	\rho^\e_{f^k}, \ |u^\e_{f^k}| \leq \frac1{\e^{2d+1}} \quad \mbox{for all } (x,t) \in \R^d \times (0,T)
	\eq
	to estimate 
	\[
	\left(1+|v|^q\right) M^\e[f^k](v) \leq 1+ |u^\e_{f^k}|^q + (c_d\rho^\e_{f^k})^\frac{q}{d} \leq C_\e
	\]
	for some $C_\e > 0$ independent of $k$. This yields
	\[
	\sup_{0 \leq t \leq T}\|f^{k+1}(\cdot,\cdot,t)\|_{L^\infty_q} \leq \|f_{\e,0}\|_{L^\infty_q} + C_\e,
	\]
	where $C_\e > 0$ is independent of $k$. This completes the proof.
\end{proof}

\subsubsection{Cauchy estimates}\label{ssec_cauchy}
In this part, we prove that $f^k$ is a Cauchy sequence in $L^\infty(0,T;L^1_2(\R^d \times \R^d))$. Straightforward computation gives
\begin{align*}
&\frac{d}{dt}\intrr (1+|v|^2)  |f^{k+1} - f^k|\,dxdv + \intrr (1+|v|^2)  |f^{k+1} - f^k|\,dxdv \cr
&\quad = \intrr (1+|v|^2) {\rm sgn}(f^{k+1} - f^k)(M^\e[f^k]- M^\e[f^{k-1}])\,dxdv,
\end{align*}
where {\rm sgn} denotes the sign function. A simple calculation gives 
	$$
	\int_{\R^d}	|\rho^\e_g-\rho^\e_h|+|u_g^\e-u_h^\e|\,dx\le  C_\e	\int_{\R^d}	|\rho_g-\rho_h|+|\rho_gu_g-\rho_hu_h|\,dx\le C_\e\|g-h\|_{L^1_2}.
	$$
Then we use Lemma \ref{lipschitz}  to estimate the right hand side of the above as
\[
\intrr \lal v \ral^2 {\rm sgn}(f^{k+1} - f^k)(M^\e[f^k] - M^\e[f^{k-1}])\,dxdv \leq C_\e \|f^k - f^{k-1}\|_{L^1_2}
\]
thanks to \eqref{bdd_e}. This shows
\[
\frac{d}{dt}\|f^{k+1} - f^k\|_{L^1_2}+\|f^{k+1} - f^k\|_{L^1_2} \leq C_\e \|f^k - f^{k-1}\|_{L^1_2},
\]
where $C_\e>0$ is independent of $k$, and subsequently, $f^k$ is a Cauchy sequence in $L^\infty(0,T;L^1_2(\R^d \times \R^d))$. Thus, for a fixed $\e>0$, there exists a limiting function $f$ such that 
\[
\sup_{0 \leq t \leq T}\|(f^k - f)(\cdot,\cdot,t)\|_{L^1_2} \to 0 \quad \mbox{as } n \to \infty.
\]
From the above, we also find that $f \in L^\infty(0,T; L^\infty_q(\R^d \times \R^d))$ and 
\[
\sup_{0 \leq t \leq T}\|f(\cdot,\cdot,t)\|_{L^\infty_q} \leq \|f_{\e,0}\|_{L^\infty_q} + C_\e
\]
for some $C_\e > 0$.

We next claim that $\rho^\e_{f^k}$ and $\rho^\e_{f^k}u^\e_{f^k}$ converge to $\rho^\e_{f}$ and $\rho^\e_{f}u^\e_{f}$ as $k \to \infty$ in $L^\infty(0,T; L^p)$ for any $p\in[1,\infty)$. Note that
\bq\label{conv_est0}
\sup_{0 \leq t \leq T}\lt(\|(\rho_{f^k} - \rho_f)(\cdot,t)\|_{L^1} + \|(\rho_{f^k}u_{f^k} - \rho_f u_f)(\cdot,t)\|_{L^1}\rt) \leq 2 \sup_{0 \leq t \leq T}\|(f^k - f)(\cdot,\cdot,t)\|_{L^1_2} \to 0
\eq
as $k \to \infty$. On the other hand, we see that
\[
\rho_{h}, \ |\rho_{h}u_{h}| \leq C\|h\|_{L^\infty_q}
\]
for $q > d+1$. Then this together with Lemma \ref{lem_uni} and \eqref{conv_est0} yields that for any $p \in [1,\infty)$
\begin{align*}
&\|(\rho^\e_{f^k} - \rho^\e_f)(\cdot,t)\|_{L^p} + \|(u^\e_{f^k} - u^\e_f)(\cdot,t)\|_{L^p} \cr
&\quad \leq C_\e\lt(\|(\rho_{f^k} - \rho_f)(\cdot,t)\|_{L^1} + \|(\rho_{f^k}u_{f^k} - \rho_f u_f)(\cdot,t)\|_{L^1} \rt),
\end{align*}
and thus
\bq\label{conv_05}
\sup_{0 \leq t \leq T}\lt(\|(\rho^\e_{f^k} - \rho^\e_f)(\cdot,t)\|_{L^p} + \|(u^\e_{f^k} - u^\e_f)(\cdot,t)\|_{L^p}\rt) \to 0
\eq
as $k \to \infty$ for any $p \in [1,\infty)$. 

\subsubsection{Proof of Proposition \ref{prop_reg}}\label{ssec_pro} We now show that the limiting function $f$ satisfies the equation \eqref{reg_eqn} in the sense of Definition \ref{def_weak}. Note that for any $\eta \in \mc^1_c(\R^d \times \R^d \times [0,T])$ with $\eta(x,v,T) = 0$, $f^{k+1}$ satisfies
\begin{align*}
&- \intrr f_{\e,0} \eta_0\,dxdv - \int_0^T \intrr f^{k+1} (\pa_t \eta + v \cdot \nabla_x \eta)\,dxdvdt \cr
&\quad = \int_0^T \intrr \lt(M^\e[f^k] - f^{k+1}\rt)\eta\,dxdvdt.
\end{align*}
Since the left hand side and the second term on the right hand side  of the above are linear, we only need to show that  
\bq\label{conv_est1}
\int_0^T \intrr M^\e[f^k] \eta\,dxdvdt \to \int_0^T \intrr M^\e[f] \eta\,dxdvdt \quad \mbox{as} \quad k \to \infty.
\eq
It follows from Lemma \ref{lipschitz} that
	\begin{align*}
	\int_0^T \intrr |M^\e[f^k]-M^\e[f] |\eta\,dxdvdt&\le  \|\eta\|_{L^\infty}\int_0^T\intrr|M^\e[f^k]-M^\e[f]|\,dxdvdt\cr
	&\le C\int_0^T\|\rho^{\e}_{f^k}-\rho^\e_f\|_{L^1}+\|u^{\e}_{f^k}-u^\e_f\|_{L^1}\,dt,
	\end{align*}
	where we used the fact that $\rho^{\e}_{(\cdot)}$ and $u^{\e}_{(\cdot)}$ are bounded from above by $1/\e^{2d+1}$. Thus the assertion \eqref{conv_est1} holds due to \eqref{conv_05}. Hence the existence of the weak solution $f$ to the equation \eqref{reg_eqn} is proved. 

It now remains to show that $f$ satisfies the kinetic energy estimate \eqref{kin_en}. Similarly as before, we get
\[
\intrr |v|^2 f^{k+1}(x,v,t)\,dxdv + \int_0^t \intrr |v|^2\lt(f^{k+1} - M^\e[f^k]\rt)\,dxdvds\leq \intrr |v|^2 f_{\e,0}(x,v)\,dxdv 
\]
for $t \in [0,T]$, and thus it suffices to show
\[
\int_0^t \intrr |v|^2M^\e[f^k]\,dxdvds \to \int_0^t \intrr |v|^2M^\e[f]\,dxdvds \quad \mbox{as} \quad k \to \infty. 
\]
Note that 
\[
\intr |v|^2(M^\e[f^k] - M^\e[f])\,dv = \rho^\e_{f^k}|u^\e_{f^k}|^2 - \rho^\e_{f}|u^\e_{f}|^2 + \frac{\left|\mathbb{S}_{d-1}\right|}{d+2}\left(\frac{d}{\left|\mathbb{S}_{d-1}\right|}\right)^{\frac{d+2}{d}}\left\{(\rho^\e_{f^k})^{1 + \frac2d} - (\rho^\e_{f})^{1 + \frac2d}\right\}.
\]
Since $\rho^\e_{(\cdot)}$and $|u^\e_{(\cdot)}|$ are bounded from above by $1/\e^{2d+1}$, we deduce
\[
\intrr |v|^2(M^\e[f^k] - M^\e[f])\,dxdv \leq C_\e\lt(\|\rho^\e_{f^k} - \rho^\e_f \|_{L^1} + \|u^\e_{f^k} - u^\e_{f}\|_{L^1}  \rt)
\]
for some $C_\e > 0$ independent of $k$. Thus, by \eqref{conv_05}, we conclude the desired result.

We now show the uniform-in-$\e$ bound estimates on $f_\e$ presented in Proposition \ref{prop_reg}. First it is clear from Lemma \ref{lem_uni} that 
\[
\sup_{0 \leq t \leq T}\|f_\e(\cdot,\cdot,t)\|_{L^1 \cap L^\infty} \leq \|f_{\e,0}\|_{L^1 \cap L^\infty} +1.
\]
For the uniform kinetic energy estimate, we observe that 
\[
\|\rho^\e_{f_\e}\|_{L^p} \leq \|\rho_{f_\e} * \theta^\e\|_{L^p} \leq \|\rho_{f_\e}\|_{L^p}
\]
for any $p \geq 1$, and
\[
\rho^\e_{f_\e} |u^\e_{f_\e}|^2 \leq \frac{|(\rho_{f_\e} u_{f_\e})*\theta^\e|^2}{\rho_{f_\e} * \theta^\e} \leq (\rho_{f_\e} |u_{f_\e}|^2)*\theta^\e,
\]
and thus
$$
\intr \rho^\e_{f_\e} |u^\e_{f_\e}|^2\,dx \leq \intr \rho_{f_\e} |u_{f_\e}|^2\,dx.
$$
This implies
\[
\intrr |v|^2 M^\e[f_\e]\,dxdv \leq \intrr |v|^2 M[f_\e]\,dxdv. 
\]
Hence, due to the minimization principle,
\[
\frac d{dt}\intrr |v|^2 f_\e\,dxdv = \intrr |v|^2 (M^\e[f_\e] - f_\e)\,dxdv \leq  \intrr |v|^2 (M[f_\e] - f_\e)\,dxdv \leq 0,
\]
which gives
\[
\intrr |v|^2 f_\e(x,v,t)\,dxdv \leq \intrr |v|^2 f_{\e,0}\,dxdv.
\]
This completes the proof.

%
%
%
%

\subsection{Proof of Theorem \ref{main_thm}: $\gamma = \frac{d+2}d$ case}\label{sec_main} We now pass to the limit $\e \to 0$ and show that $f:=\lim_{\e \to 0}f_\e$ satisfies the equation \eqref{BGK} in the sense of Definition \ref{def_weak} and the kinetic energy inequality \eqref{kin_ineq}.

We first present a lemma, showing some relationship between the local density and the kinetic energy, which will be used to estimate the interaction energy. Although its proof is by now classical (see \cite[Lemma 3.1]{GS99} for instance), for the reader's convenience, we provide the details of it. 
\begin{lemma}\label{lem_tech} Suppose $f \in L^1_+ \cap L^\infty(\R^d \times \R^d)$  and $|v|^2 f \in L^1(\R^d \times \R^d)$. Then there exists a constant $C>0$ such that
\[
\|\rho_f\|_{L^\frac{d+2}d} \leq C\left\|f\right\|_{L^\infty}^{ \frac{d+2}2} \lt(\intrr |v|^2 f\,dxdv\rt)^{\frac d{d+2}}
\]
and
\[
\|\rho_f u_f\|_{L^\frac{d+2}{d+1}} \leq C\left\|f\right\|_{L^\infty}^{d+2} \lt(\intrr |v|^2 f\,dxdv\rt)^{\frac {d+1}{d+2}}.
\]
\end{lemma}
\begin{proof} Note that for any $R>0$
\[
\rho_f = \intr f\,dv = \lt(\int_{|v| \geq R} + \int_{|v| \leq R} \rt) f\,dv \leq \frac1{R^2} \intr |v|^2 f\,dv + C\left\|f\right\|_{L^\infty} R^d.
\]
We now take $R = \lt(\intr |v|^2 f\,dv / \left\|f\right\|_{L^\infty} \rt)^{1/{d+2}}$ to obtain
\[
\rho_f \leq C\left\|f\right\|_{L^\infty}^{ \frac{d+2}2} \lt(\intr |v|^2 f\,dv \rt)^{\frac d{d+2}}.
\]
This implies
\[
\|\rho_f\|_{L^\frac{d+2}{d}} \leq C \left\|f\right\|_{L^\infty}^{ \frac{d+2}2}  \lt(\intrr |v|^2 f\,dxdv\rt)^{\frac d{d+2}}.
\]
Similarly, we also have the second assertion.
\end{proof}

Using the uniform bound estimates and Lemma \ref{lem_tech}, we obtain that there exists $f \in L^\infty (\R^d \times \R^d \times (0,T))$ such that
\[
f_\e \stackrel{\ast}{\rightharpoonup} f \quad \mbox{in } L^\infty (\R^d \times \R^d \times (0,T)), \quad \rho_{f_\e} \stackrel{\ast}{\rightharpoonup} \rho_f \quad \mbox{in } L^\infty (0,T; L^p(\R^d)) \quad \mbox{with } p \in \lt[1,\frac{d+2}{d}\rt],
\]
and
\[
 \rho_{f_\e} u_{f_\e} \stackrel{\ast}{\rightharpoonup} \rho_f u_f \quad \mbox{in } L^\infty (0,T; L^q(\R^d)) \quad \mbox{with } q \in \lt[1,\frac{d+2}{d+1}\rt]
\]
as $\e \to 0$.  On the other hand, we know $M^\e[f_\e] \in L^\infty(\R^d \times \R^d \times (0,T))$, and thus $M^\e[f_\e] \in L^p_{loc}(\R^d \times \R^d \times (0,T))$. Moreover, 
\begin{align}\label{sp_mom}
\begin{aligned}
&\frac{d}{dt}\intrr |x|^2 f_\e\,dxdv + \intrr |x|^2 f_\e\,dxdv \cr
&\quad = 2\intrr x \cdot v f_\e\,dxdv + \intrr |x|^2 M^\e[f_\e]\,dxdv\cr
&\quad \leq 3\intrr |x|^2 f_\e\,dxdv + \intrr |v|^2 f_\e\,dxdv + 2\e,
\end{aligned}
\end{align}
where we used
\[
\intrr |x|^2 M^\e[f_\e]\,dxdv = \intr |x|^2 \rho^\e_{f_\e}\,dx \leq \intr |x|^2 (\rho_{f_\e} * \theta^\e)\,dx \leq 2\intr |x|^2 \rho_{f_\e}\,dx + 2\e.
\]
Since the kinetic energy is uniformly bounded in $\e>0$, applying the Gr\"onwall's lemma to \eqref{sp_mom} gives
\[
\intrr |x|^2 f_\e\,dxdv < \infty
\]
uniformly in $\e>0$. 

We next recall the following velocity averaging lemma, whose proof can be found in \cite[Lemma 2.7]{KMT13}.
\begin{lemma}
Let $\{f^m\}$ be bounded in $L_{loc}^p(\R^d \times \R^d \times [0,T])$ with $1 < p < \infty$ and $\{G^m\}$ be bounded in $L_{loc}^p(\R^d \times \R^d \times [0,T])$. Suppose that
\[
\sup_{m \in \bbn} \|f^m\|_{L^\infty(\R^d \times \R^d \times (0,T))} + \sup_{m \in \bbn}\| (|v|^2 + |x|^2) f^m\|_{L^\infty(0,T;L^1(\R^d \times \R^d))}  <\infty.
\]
 If $f^m$ and $G^m$ satisfy
\[
\pa_t f^m + v \cdot \nabla_x f^m = \nabla_v^\ell G^m, \quad f^m|_{t=0} = f_0 \in L^p(\R^d \times \R^d),
\]
then for any $\eta(v)$ satisfying $|\eta(v)| \ls 1+ |v|$, the sequence
\[ \left\{ \int_{\R^d} f^m \eta \,dv \right\}_m \]
is relatively compact in $L^q(\R^d \times (0,T))$ for any $q \in \left(1, \frac{d+2}{d+1} \right)$.
\end{lemma}

Then, by a direct application of the above lemma, we have the following convergences:
\bq\label{strong_comp}
\rho_{f_\e} \to \rho_f \quad  \mbox{in } L^\infty (0,T; L^p(\R^d)) \quad \mbox{and} \quad   \rho_{f_\e} u_{f_\e} \to \rho_f u_f \quad  \mbox{in } L^\infty (0,T; L^p(\R^d))
\eq
for $p \in (1,\frac{d+2}{d+1})$ as $\e \to 0$. On the other hand, it follows from \eqref{strong_comp} that
\bq\label{conv_ae0}
\|\rho_{f_\e} * \theta^\e - \rho_f\|_{L^1} \leq \|(\rho_{f_\e} - \rho_f)*\theta^\e\|_{L^1} + \|\rho_{f} * \theta^\e - \rho_f\|_{L^1} \to 0 \quad \mbox{as} \quad \e \to 0,
\eq
Similarly $\|(\rho_{f_\e} u_{f_\e}) * \theta^\e - \rho_f u_f\|_{L^1}$ as $\e \to 0$. Thus combining that with 
\[
\e^{d+1}|\rho_{f_\e} * \theta^\e| \leq C\e \quad \mbox{and} \quad \e^{2d+1}|(\rho_{f_\e}u_{f_\e}) * \theta^\e| \leq C\e
\]
deduces that 
\bq\label{conv_ae}
\rho^\e_{f_\e} \to \rho_f  \quad \mbox{and} \quad u^\e_{f_\e} \to u_f  \quad \mbox{a.e. on } E, 
\eq
where $E := \{(x,t)\in \R^d\times [0,T] : \rho_f(x,t) > 0\}$. 
 Indeed,
\begin{align*}
|\rho^\e_{f_\e} - \rho_f |&=\left|\frac{\rho_f*\theta^\e -(1+\e^{d+1}\rho_f*\theta^\e)\rho_f}{1+\e^{d+1}\rho_f*\theta^\e}\right|\cr
&=\left|\frac{\rho_f*\theta^\e-\rho_f -\e\rho_f(\rho_f*\e^d\theta^\e)}{1+\e(\rho_f*\e^d\theta^\e)}\right|\to 0\quad \mbox{a.e.}\quad \mbox{on}\quad E
\end{align*}
due to \eqref{conv_ae0}. We now show that the limiting function $f$ satisfies the weak formulation of  our main equation \eqref{BGK}, and again for this, it suffices to provide the following convergence:
 \begin{equation*} 
 \int_0^T \intrr M^\e[f_\e] \eta\,dxdvdt \to \int_0^T \intrr M[f] \eta\,dxdvdt \quad \mbox{as} \quad \e \to 0
\end{equation*} 
for all $\eta \in \mc_c^1(\R^d \times \R^d \times [0,T])$  with $\eta(x,v,T) = 0$. For this, we follow the idea of \cite{Per89}. 

Recall that 
$$
M^\e[f_\e]= \mathbf{1}_{|u^\e_{f_\e}-v|^d\le c_d\rho^\e_{f_\e}},\quad M[f]= \mathbf{1}_{|u^\e_{f_\e}-v|^d\le c_d\rho^\e_{f_\e}}.
$$
We deduce from \eqref{conv_ae} that for each $v$ belonging to the closure of $B(u_f, (c_d\rho_f)^{\frac 1d})$,
\begin{equation*} 
M^\e[f_\e](v) \to 1=M[f](v),\quad \mbox{and otherwise,}\quad  M^\e[f_\e](v) \to 0\quad\mbox{a.e. on }E\quad \mbox{as}\quad \e \to 0,
\end{equation*}
i.e. $M^\e[f_\e](v)$ converges to $M[f]$ a.e. on $E\times \mathbb{R}^d$. Moreover, we have from the $L^1$-bound of  $f^\K$ that for any $p\in(1,\infty)$,
\begin{align*}
\|M^\e[f_\e]\|_{L^p}^p&=\iint_{\mathbb{R}^d\times \mathbb{R}^d} M^\e[f_\e]\,dxdv= \int_{\mathbb{R}^d} \rho_{f_\e}\,dx <C.
\end{align*}
On the other hand, on $E^c \times \R^d$, we estimate
\begin{align*}
\lt|\lim_{\e \to 0}\iiint_{E^c \times \R^d} M^\e[f_\e] \eta \,dxdvdt \rt| &\leq \|\eta\|_{L^\infty} \lim_{\e \to 0}  \iiint_{E^c \times \R^d} M^\e[f_\e] \,dxdvdt \cr
&\leq\lim_{\e \to 0}  \iint_{E^c} \rho_{f_\e}\,dxdt\cr
&=\iint_{E^c} \rho_f \,dxdt \cr
&= 0.
\end{align*}
Thus one can conclude that  
\begin{align*}
\lim_{\e \to 0}\int_0^T\iint_{\R^d \times \R^d}  M^\e[f_\e] \eta\,dxdvdt &= \lim_{\e \to 0} \iiint_{E \times \R^d}  M^\e[f_\e] \eta\,dxdvdt + \lim_{\e \to 0} \iiint_{E^c \times \R^d}  M^\e[f_\e] \eta\,dxdvdt\cr
&= \iiint_{E \times \R^d}  M[f] \eta\,dxdvdt\cr
&= \int_0^T\iint_{\R^d \times \R^d}  M[f] \eta\,dxdvdt.
\end{align*}
This completes the proof.

%
%
%
%
%

\section{Proof of Theorem \ref{main_thm}: $\gamma \in (1,\frac{d+4}{d+2}]$ case  }\label{sec_4}

In this section, we provide the details of the proof of Theorem \ref{main_thm} in the case $\gamma \in (1,\frac{d+4}{d+2}]$ when  $d\geq 2$.  The proof can be directly applied to the one-dimensional case, i.e., $\gamma \in (1,3)$. Since the main idea of proof is essentially the same with the case $\gamma = \frac{d+2}d$, here we give a rather brief outline of the proof. As mentioned before, due to technical difficulties, we only consider the case of periodic spatial domain.

Similarly as in Section \ref{sec_bcase}, we regularize the equilibrium function $M[f]$ by employing the regularized macroscopic quantities $\rho^\e_f$ and $u^\e_f$ appeared in \eqref{reg_fcts}. 
Then our regularized equation of \eqref{BGK} reads as
\begin{align}\label{reg_eqn2}
\begin{aligned}
\pa_t f_\e + v \cdot \nabla_x f_\e = M^\e[f_\e] - f_\e, \quad \mbox{where} \quad  M^\e[f_\e] = c\left(\frac{2\gamma}{\gamma-1}(\rho^\e_{f_\e})^{\gamma-1}-|v-u^\e_{f_\e}|^2\right)^{n/2}_+,
\end{aligned}
\end{align}
subject to the regularized initial data $f_{\e,0}$ given as in \eqref{reg_ini}. 

Parallel to Section \ref{sec_bcase}, we first provide the global-in-time existence of weak solutions to the regularized  equation \eqref{reg_eqn2}.

\begin{proposition}\label{prop_reg2} Let $T>0$ and the initial data $f_0$ satisfy the assumptions of Theorem \ref{main_thm}. Then there exists a weak solution $f_\e$ of the equations \eqref{reg_eqn2} with \eqref{reg_ini} in the sense of Definition \ref{def_weak}. Moreover, we have
	\[
	\sup_{0 \leq t \leq T}\|f_\e(\cdot,\cdot,t)\|_{L^1} \leq \|f_{\e,0}\|_{L^1 }  \quad \mbox{and} \quad \rho_{f_\e}(x,t) \geq c_\e 
	\]
	for some $c_\e > 0$, 	and
	\begin{align}\label{kin_en2}\begin{split}
	&\inttr H(f_\e,v)\,dxdv + \int_0^t \inttr H(f_\e,v)-H(M[f_\e],v)\,dxdvds\leq \inttr H(f_{0,\e},v)\,dxdv 
\end{split}	\end{align}
	for $t \in [0,T]$. 
	In particular, this implies
	\begin{equation}\label{uniform bound}
\sup_{0\le t\le T}	\inttr \frac{|v|^2}{2} f_\e+\frac{1}{2c^{2/n}}\frac{( f_\e)^{1+2/n}}{1+2/n}\,dxdv \leq \inttr \frac{|v|^2}{2} f_{0,\e}+\frac{1}{2c^{2/n}}\frac{( f_{0,\e})^{1+2/n}}{1+2/n}\,dxdv
	\end{equation}
	due to the minimization principle \eqref{minimization}.
\end{proposition}

In order to prove Proposition \ref{prop_reg2}, we consider the approximation sequence $f^k_\e$ satisfying
\bq\label{app_weak_eq2}
\pa_t f_\e^{k+1} + v\cdot\nabla_x f_\e^{k+1}   = M^\e[f_\e^k] - f_\e^{k+1}, \quad \mbox{where} \quad  M^\e[f_\e^k] = c\left(\frac{2\gamma}{\gamma-1}(\rho^\e_{f_\e^k})^{\gamma-1}-|v-u^\e_{f_\e^k}|^2\right)^{n/2}_+
\eq
with the initial data and first iteration step:
\[
f_\e^k(x,v,t)|_{t=0} = f_{\e,0}(x,v) \quad \mbox{for all } k \geq 1 \quad \mbox{and} \quad f_\e^0(x,v,t) = f_{\e,0}(x,v) \quad \mbox{for } (x,v,t) \in \T^d \times \R^d \times (0,T).
\]

For notational simplicity, we often drop the subscript $\e$ and denote $f^k_\e$ by $f^k$ for instance.

\begin{lemma}\label{lem_uni2} Let $T>0$ and $q > d+1$. Suppose that the initial data $f_0$ satisfies the assumptions of Theorem \ref{main_thm}. For any $k \in \N$, there exists a unique solution $f^k$ of the equation \eqref{app_weak_eq2} such that $f^k \in L^\infty(0,T; L^\infty_q(\T^d \times \R^d))$. Moreover, we have
	\[
	\sup_{k \in \N}\sup_{0 \leq t \leq T}\left(\|  f^k(\cdot,\cdot,t)\|_{L^1_2}+\|f^k(\cdot,\cdot,t)\|_{L^{1+\frac 2n}}\right) \leq C , \qquad \inf_{k \in \N}\rho_{f^k}(x,t) \geq c_{\e} 
	\]
	and
	\[
	\sup_{k \in \N}\sup_{0 \leq t \leq T}\|f^k(\cdot,\cdot,t)\|_{L^\infty_q} \leq \|f_{\e,0}\|_{L^\infty_q} + C_\e,
	\]
	for some $c_\e > 0$ and $C_\e > 0$ independent of $k$.
\end{lemma}
\begin{proof} Note that by the classical existence theory, we obtain the existence and uniqueness of solutions $f^k \in L^\infty(0,T; L^\infty_q(\T^d \times \R^d))$ to \eqref{app_weak_eq2}. Regarding the bound estimates on $f^k$, we first employ the same argument as in the proof of Lemma \ref{lem_uni} to obtain $\|f^k(t)\|_{L^1} \leq \|f_{\e,0}\|_{L^1}$ for all $k \geq 0$ and $t \in [0,T]$. We notice that the equilibrium function $M[f]$ is bounded by $1$ in the case $\gamma = \frac{d+2}{d}$, however, in the current case, we cannot simply bound it by some constant. In fact, to get the global-in-time existence of solutions, it is important to control the equilibrium function at least linearly in terms of $\|f^k\|_{L^1\cap L^r}$ for some $r>\gamma$. For this, differently from the argument used in the proof of Lemma \ref{lem_uni}, we make use of the kinetic entropy \eqref{k_entropy} and the minimization principle \eqref{minimization}.
	
We first estimate $\|f^k(\cdot,\cdot,t)\|_{L^1_2}+\|f^k(\cdot,\cdot,t)\|_{ L^{1+\frac 2n}}$. It follows from \eqref{moment_comp} (see also Lemma \ref{moments}) that
	$$
	\inttr M^\e[f^k]\,dxdv=\int_{\T^d}  \rho^\e_{f^k}\,dx\le \int_{\T^d}  \rho_{f^k}*\theta_\e\,dx\le \int_{\T^d}  \rho_{f^k}\,dx,
	$$
and thus 
$$
\frac{d}{dt}\inttr f^{k+1}\,dxdv\le \inttr f^{k}\,dxdv-\inttr f^{k+1}\,dxdv
$$
which implies $\|f^k(t)\|_{L^1}\le \|f_{\e,0}\|_{L^1}$ for all $k\ge 0$ and $t\in[0,T]$. For the estimate of $L^{1}$-norm of $|v|^2f^{k}$ and $L^{1+\frac2n}$-norm of $f^{k}$, we observe that
	$$
	\frac{d}{dt}	\inttr |v|^2 f^{k+1}\,dxdv =\inttr |v|^2 M^\e[f^k]\,dxdv-\inttr |v|^2 f^{k+1} \,dxdv
	$$
	and
	\begin{align*}
	\frac{d}{dt}\inttr(f^{k+1})^{1+\frac 2n}\,dxdv&=\frac{n+2}{n}\inttr	M^\e[f^k](f^{k+1})^{\frac 2n}\,dxdv-\frac{n+2}{n}\inttr	(f^{k+1})^{1+\frac 2n}\,dxdv.
	\end{align*}
	Using Young's inequality with $p=\frac{n+2}{n}$ and $q=\frac{n+2}{2}$, one finds
	$$
	\frac{d}{dt}\inttr(f^{k+1})^{1+\frac 2n}\,dxdv\le\inttr	(M^\e[f^k])^{1+\frac 2n}\,dxdv-\inttr	(f^{k+1})^{1+\frac 2n}\,dxdv.
	$$
	Combining these results, we deduce
		\begin{equation}\label{L gamma}
	\frac{d}{dt}\inttr H(f^{k+1},v)\,dxdv\le\inttr	H(M^\e[f^k],v)\,dxdv-\inttr	H(f^{k+1},v)\,dxdv.
	\end{equation}
	Note from Lemma \ref{moments} that
	\begin{align}\label{H in the middle}\begin{split}
	\intr	H(M^\e[f^k],v)\,dv&=\frac 12\rho_{f^k}^\e|u_{f^k}^\e|^2+\frac 12(\rho_{f^k}^\e)^\gamma+\frac{1}{2c^{2/n}}\frac{c^{1+\frac 2n}}{1+\frac 2n}\intr \left(\frac{2\gamma}{\gamma-1}(\rho^\e_{f_\e})^{\gamma-1}-|v-u^\e_{f_\e}|^2\right)^{\frac{n+2}{2}}_+\, dv\cr
	&=\frac 12\rho_{f^k}^\e|u_{f^k}^\e|^2+\frac 12(\rho_{f^k}^\e)^\gamma+\frac{c}{2+\frac 4n}\left(\frac{2\gamma}{\gamma-1}(\rho^\e_{f_\e})^{\gamma-1}\right)^{\frac{n+2}{2}+\frac d2}\intr \left(1-|w|^2\right)^{\frac{n+2}{2}}_+\, dw\cr
	&=\frac 12\rho_{f^k}^\e|u_{f^k}^\e|^2+\frac 32(\rho_{f^k}^\e)^\gamma.
\end{split}	\end{align}
In the last line, we have followed the same line of the proof of Lemma \ref{moments}, and used the following relation:
$$
B(x,y+1)=B(x,y)\frac{y}{x+y}.
$$
Then one can see that	
\begin{align}\label{entropy gamma} \begin{split}
\inttr	H(M^\e[f^k],v)\,dxdv&= \int_{\T^d}  \frac 12\rho_{f^k}^\e|u_{f^k}^\e|^2+\frac 32(\rho_{f^k}^\e)^\gamma\,dx\cr
&\le \int_{\T^d}  \frac 12(\rho_{f^k}|u_{f^k}|^2)*\theta^\e+\frac 32(\rho_{f^k}*\theta^\e)^\gamma\,dx\cr
&\le \int_{\T^d} \frac 12\rho_{f^k}|u_{f^k}|^2+\frac 32(\rho_{f^k})^\gamma\,dx\cr
&=\inttr	H(M[f^k],v)\,dxdv
\end{split}\end{align}
which together with the minimization principle \eqref{minimization} implies
\begin{equation}\label{H Me}
\inttr	H(M^\e[f^k],v)\,dxdv \le \inttr	H(f^k,v)\,dxdv.
\end{equation}
Thus, going back to \eqref{L gamma}, we conclude that  
	$$
\inttr H(f^{k},v)\,dxdv\le \inttr H(f_{0,\e},v)\,dxdv
$$
for any $k\in \mathbb{N}$ and $\gamma\in (1,\frac{d+2}{d})$, which gives the first assertion. 

The lower bound estimate on $\rho^k$ is exactly the same with that of Lemma \ref{lem_uni}, but we now have a strictly positive lower bound due to $x \in \T^d$.

We finally show the bound estimate of $\|f^k\|_{L^\infty_q}$ for $q > d+1$. It follows from  \eqref{bdd_e} that 
\begin{align*}
\left(1+|v|^q\right) M^\e[f^k](v)&=\left(1+|v|^q\right) c	\left(\frac{2\gamma}{\gamma-1}(\rho^\e_{f^k})^{\gamma-1}-|v-u^\e_{f^k}|^2\right)^{\frac n2} \mathbf{1}_{|v-u^\e_{f^k}|^2\le \frac{2\gamma}{\gamma-1}(\rho^\e_{f^k})^{\gamma-1}}\cr
& \leq C(\rho^\e_{f^k})^{\frac{n(\gamma-1)}{2}}\left(1+|u^\e_{f^k}|+(\rho^\e_{f^k})^{\frac{\gamma-1}{2}} \right)^q   \cr
& \leq C_\e
\end{align*}
for some $C_\e > 0$ independent of $k$. This yields
\[
\sup_{0 \leq t \leq T}\|f^{k+1}(\cdot,\cdot,t)\|_{L^\infty_q} \leq \|f_{\e,0}\|_{L^\infty_q} + C_\e,
\]
where $C_\e > 0$ is independent of $k$. This completes the proof.
\end{proof}

\subsection{Proof of Proposition \ref{prop_reg2}}
By using almost the same argument as in Section \ref{ssec_cauchy} and Lemma \ref{lipschitz} in the case $\gamma \in (1,\frac{d+4}{d+2}]$, we have that there exists a limiting function $f\in L^\infty(0,T; L^\infty_q(\T^d \times \R^d))$ such that 
\bq\label{conv_050}
\sup_{0 \leq t \leq T}\|(f^k - f)(\cdot,\cdot,t)\|_{L^1_2} \to 0 
\eq
and
\bq\label{conv_051}
\sup_{0 \leq t \leq T}\lt(\|(\rho^\e_{f^k} - \rho^\e_f)(\cdot,t)\|_{L^p} + \|(u^\e_{f^k} - u^\e_f)(\cdot,t)\|_{L^p}\rt) \to 0
\eq
as $k \to \infty$ for any $p \in [1,\infty)$. We now claim that the limiting function $f$ obtained in the above satisfies the equation \eqref{reg_eqn2} in the sense of distributions. For this, it suffices to show that 
$$
\int_0^T \inttr M^\e[f^k] \eta\,dxdvdt \to \int_0^T \inttr M^\e[f] \eta\,dxdvdt \quad \mbox{as} \quad k \to \infty
$$
for any $\eta \in \mc^1_c(\T^d \times \R^d \times [0,T])$. 
It follows from Lemma \ref{lipschitz} that
	\begin{align*}
	\int_0^T \inttr |M^\e[f^k]-M^\e[f] |\eta\,dxdvdt&\le  \|\eta\|_{L^\infty}\int_0^T\inttr|M^\e[f^k]-M^\e[f]| \,dxdvdt\cr
	&\le C\int_0^T\|\rho^{\e}_{f^k}-\rho^\e_f\|_{L^1}+\|u^{\e}_{f^k}-u^\e_f\|_{L^1}\,dt.
	\end{align*}
 Combining the last inequality with \eqref{conv_051} gives the desired result.

To prove the kinetic entropy inequality \eqref{kin_en2}, we use almost the same argument as in Section \ref{ssec_pro}. It follows from \eqref{L gamma} and \eqref{entropy gamma} that
\begin{align*}
&\inttr H(f^{k+1},v)\,dxdv+\int_0^t\inttr H(f^{k+1},v)\,dxdvds-\int_0^t\int_{\T^d} \frac 12\rho_{f^k}^\e|u_{f^k}^\e|^2+\frac 32(\rho_{f^k}^\e)^\gamma\,dxds \cr
&=\inttr H(f_{\e,0},v)\,dxdv,
\end{align*}
 where $H(\cdot,v)$ is the kinetic entropy defined in \eqref{k_entropy} as  
$$
H(h,v)=\frac{|v|^2}{2}h+\frac{1}{2c^{2/n}}\frac{h^{1+2/n}}{1+2/n}.
$$
Thanks to $L^1$ convergence in \eqref{conv_050} and $L^\infty_q$-estimates of $f^{k}$ and $f$, it is enough to show that
$$
\int_0^t\int_{\T^d}  \rho_{f^k}^\e|u_{f^k}^\e|^2+3(\rho_{f^k}^\e)^\gamma\,dxds\rightarrow \int_0^t\int_{\T^d} \rho_{f}^\e|u_{f}^\e|^2+3(\rho_{f}^\e)^\gamma\,dxds\quad\mbox{as}\quad n\rightarrow \infty.
$$
Since $\rho^\e_{(\cdot)}$ and $u^\e_{(\cdot)}$ are bounded from above by $1/\e^{2d+1}$, and $\gamma$ is bigger than $1$, we deduce
\begin{align*}
\int_{\T^d}  \left\{\rho_{f^k}^\e|u_{f^k}^\e|^2+3(\rho_{f^k}^\e)^\gamma\right\}- \left\{\rho_{f}^\e|u_{f}^\e|^2+3(\rho_{f}^\e)^\gamma\right\}\,dx\le C_\e\left(\|\rho_{f^k}^\e-\rho_{f}^\e\|_{L^1}+\|u_{f^k}^\e-u_{f}^\e\|_{L^1} \right),
\end{align*}
which together with \eqref{conv_051} implies
$$
\inttr H(f,v)\,dxdv +\int_0^t\inttr H(f,v)-H(M^\e[f],v)\,dxdvds=\inttr H(f_{\e,0},v) \,dxdv.
$$

Finally, combining the above identity with \eqref{entropy gamma} and the minimization principle \eqref{minimization}, we obtain the uniform-in-$\e$ bound estimates on $f_\e$ presented in Proposition \ref{prop_reg2}.

\subsection{Proof of Theorem \ref{main_thm}: $\gamma \in (1,\frac{d+4}{d+2}]$ case}

Due to \eqref{sp_mom} and \eqref{uniform bound}, one finds
$$
\int_0^T\inttr (1+|v|^2+f_\e^{\frac2n}) f_\e \,dxdvdt\le C.
$$	
Applying Dunford-Pettis theorem, we then have that there exists  $f \in L^1 (\T^d \times \R^d \times (0,T))$ such that
\begin{align*}
f_\e & \rightharpoonup f \quad\hspace{5.3mm} \mbox{in } L^1 (\T^d \times \R^d \times (0,T)),  \cr
\rho_{f_\e} &\rightharpoonup \rho_f \quad\hspace{3.8mm}  \mbox{in } L^1 (\T^d \times (0,T)), \quad \mbox{and}\cr
\rho_{f_\e} u_{f_\e} &\rightharpoonup \rho_f u_f \quad  \mbox{in } L^1 (\T^d \times (0,T)).
\end{align*}
This, together with the velocity averaging lemma introduced in \cite{GLPS88,Per89}, implies  that for $p \in \lt[1,\frac{n+2}{n}\rt]$,
\begin{align*}
f_\e &\rightarrow f \quad \hspace{5.3mm}\mbox{in } L^p (\T^d \times \R^d \times (0,T)),\cr
\rho_{f_\e} &\rightarrow \rho_f \quad \hspace{3.8mm}\mbox{in } L^1 (\T^d \times (0,T)), \quad \mbox{and}\cr
\rho_{f_\e} u_{f_\e} &\rightarrow \rho_f u_f \quad \mbox{in } L^1 (\T^d \times (0,T)).
\end{align*}
Thus, we deduce
\begin{equation}\label{a.e. conv}
\rho^\e_{f_\e} \to \rho_f,  \quad \mbox{and} \quad u^\e_{f_\e} \to u_f  \quad \mbox{a.e. on } E,
\end{equation}
 where $E := \{(x,t)\in \T^d\times [0,T] : \rho_f(x,t) > 0\}$. We then show that the limiting function $f$ is the weak solution to \eqref{BGK}, and again for this, it is sufficient to obtain the following convergence:
\bq\label{conv_MM}
\int_0^T \inttr M^\e[f_\e] \eta\,dxdvdt \to \int_0^T \inttr M[f] \eta\,dxdvdt \quad \mbox{as} \quad \e \to 0
\eq
for all $\eta \in \mc_c^1(\T^d \times \R^d \times [0,T])$  with $\eta(x,v,T) = 0$.  Recall that
\[
M^\e[f_\e] = c\left(\frac{2\gamma}{\gamma-1}(\rho^\e_{f_\e})^{\gamma-1}-|v-u^\e_{f_\e}|^2\right)^{n/2}\mathbf{1}_{|v-u^\e_{f_\e}|^2\le \frac{2\gamma}{\gamma-1}(\rho^\e_{f_\e})^{\gamma-1}}
\]
and
\[
M[f] = c\left(\frac{2\gamma}{\gamma-1}(\rho_{f})^{\gamma-1}-|v-u_{f}|^2\right)^{n/2}\mathbf{1}_{|v-u_{f}|^2\le \frac{2\gamma}{\gamma-1}(\rho_{f})^{\gamma-1}}.
\]
Due to \eqref{a.e. conv}, one can see that $M^\e[f_\e](v)$ converges to $M[f]$ a.e. on $E\times \mathbb{R}^d$. Moreover, it follows from \eqref{kin_en2}, \eqref{H in the middle}, and \eqref{H Me} that
\begin{align*}
\|M^\e[f_\e]\|_{L^{1+2/n}}^{1+2/n}&=c^{\frac 2n}\left(1+\frac 2n\right)\int_{ \T^d}\rho_{f_\e}^\gamma \,dx<C \inttr H(f_\e,v)\,dxdv <C.
\end{align*}
We then follow the same argument as in Section \ref{sec_main} to have the convergence \eqref{conv_MM}. This completes the proof.
%
%
%
%
%

\appendix
 
\section{Appendix}
In this appendix, we present the estimates on the first three moments of the equilibrium function $M[f]$ \eqref{Maxwellian}. Recall 
$$
M [f]= \begin{cases}
 \mathbf{1}_{|u_f-v|^d\le c_d\rho_f}\qquad &\text{for}\quad \gamma=\frac{d+2}{d},\\[2mm]
 c\left(\frac{2\gamma}{\gamma-1}\rho_f^{\gamma-1}-|v-u_f|^2\right)^{n/2}_+\qquad &\text{for} \quad \gamma \in (1,\frac{d+2}d).
\end{cases}
$$
\begin{lemma}\label{moments} Let $\gamma \in (1, \frac{d+2}d]$. Then we have
$$
	\int_{\mathbb{R}^d}(1,v,|v|^2)  M[f]\,dv =
	\left(\rho_f,\rho_fu_f, 
	\rho_f |u_f|^2 +dC_d \rho_f^\gamma \right),
	$$
where $C_d > 0$ is given by 
$$
C_d = \begin{cases}
\frac{\left|\mathbb{S}_{d-1}\right|}{d(d+2)}\left(\frac{d}{\left|\mathbb{S}_{d-1}\right|}\right)^{\frac{d+2}{d}} \qquad &\text{for}\quad \gamma=\frac{d+2}{d},\\[2mm]
1\qquad &\text{for} \quad \gamma \in (1,\frac{d+2}d).
\end{cases}
$$
\end{lemma}
\begin{proof} We divide the proof into two cases: $\gamma = \frac{d+2}d$ and $\gamma \in (1, \frac{d+2}d$).

	\vspace{.2cm}

$\bullet$  Case $\gamma = \frac{d+2}d$: By the change of variables and spherical coordinates, we obtain
	\begin{align*}
	\int_{\mathbb{R}^d} M[f]\,dv &=\int_{\mathbb{R}^d} \mathbf{1}_{|u_f-v|^d\le c_d\rho_f}\,dv \cr
	&=\int_{0}^{2\pi}\int_{0}^{\pi}\cdots\int_{0}^{\pi}\int_{0}^{(c_d\rho_f)^{1/d}}r^{d-1}(\sin^{d-2}\varphi_1) (\sin^{d-3}\varphi_{2})\cdots (\sin\varphi_{d-2}) \,drd\varphi_1\cdots d\varphi_{d-1}\cr
	&=\frac{c_d\rho_f}{d} \left|\mathbb{S}_{d-1}\right|\cr
	&=\rho_f.
	\end{align*}
	This together with the oddness gives
	\begin{align*}
	\int_{\mathbb{R}^d}v M[f]\,dv &=\int_{\mathbb{R}^d} (v+u_f)\mathbf{1}_{|v|^d\le c_d\rho_f}\,dv =u_f\int_{\mathbb{R}^d} \mathbf{1}_{|v|^d\le c_d\rho_f}\,dv=\rho_fu_f.
	\end{align*}
	Analogously, we find
		\begin{align*}
	&\int_{\mathbb{R}^d}|v|^2 M[f]\,dv \cr
	&\quad =\int_{\mathbb{R}^d} (|v|^2+2v\cdot u_f +|u_f|^2)\mathbf{1}_{|v|^d\le c_d\rho_f}\,dv \cr
	&\quad =\rho_f|u_f|^2+\int_{0}^{2\pi}\int_{0}^{\pi}\cdots\int_{0}^{\pi}\int_{0}^{(c_d\rho_f)^{1/d}}r^{d+1}(\sin^{d-2}\varphi_1) (\sin^{d-3}\varphi_{2})\cdots (\sin\varphi_{d-2})  \,drd\varphi_1\cdots d\varphi_{d-1}\cr
	&\quad =\rho_f|u_f|^2+\frac{\left|\mathbb{S}_{d-1}\right|}{d+2}\left(c_d\rho_f\right)^{\frac{d+2}{d}} \cr
	&\quad =\rho_f|u_f|^2+\frac{\left|\mathbb{S}_{d-1}\right|}{d+2}\left(\frac{d}{\left|\mathbb{S}_{d-1}\right|}\right)^{\frac{d+2}{d}}\rho_f^{\frac{d+2}{d}}.
	\end{align*}
	
		\vspace{.2cm}
	
$\bullet$ Case $\gamma \in (1, \frac{d+2}d$): By the change of variables, we first observe 
	\begin{align*}
	\int_{\mathbb{R}^d}M[f]\,dv &=c	\int_{\mathbb{R}^d}\left(\frac{2\gamma}{\gamma-1}\rho_f^{\gamma-1}-|v-u_f|^2\right)^{\frac n2} \mathbf{1}_{|v-u_f|^2\le \frac{2\gamma}{\gamma-1}\rho_f^{\gamma-1}}\,dv\cr
	&= c	\int_{\mathbb{R}^d}\left(\frac{2\gamma}{\gamma-1}\rho_f^{\gamma-1}-|v|^2\right)^{\frac n2} \mathbf{1}_{|v|^2\le \frac{2\gamma}{\gamma-1}\rho_f^{\gamma-1}}\,dv\cr
	&= c\left(\frac{2\gamma}{\gamma-1}\rho_f^{\gamma-1}\right)^{\frac{n+d}{2}}	\int_{\mathbb{R}^d}\left(1-|v|^2\right)^{\frac n2} \mathbf{1}_{|v|^2\le 1}\,dv\cr
	&= c\left(\frac{2\gamma}{\gamma-1}\rho_f^{\gamma-1}\right)^{\frac{n+d}{2}}	\left|\mathbb{S}_{d-1}\right|\int_0^1\left(1-r^2\right)^{\frac n2} r^{d-1}\,dr.
	\end{align*}
	By the definition of $c$ and $n$, we get
	\begin{align*}
	\int_{\mathbb{R}^d}M[f]\,dv &= \frac{\Gamma\left(\frac{\gamma}{\gamma-1}\right)}{\pi^{\frac d2}\Gamma(\frac n2+1)}\rho_f	\left|\mathbb{S}_{d-1}\right|\int_0^1\left(1-r^2\right)^{\frac n2} r^{d-1}\,dr\cr
	&= 2\rho_f\frac{\Gamma\left(\frac{\gamma}{\gamma-1}\right)}{\Gamma(\frac n2+1)\Gamma\left(\frac{d}{2}\right)}	\int_0^1\left(1-r^2\right)^{\frac n2} r^{d-1}\,dr,
		\end{align*}
where we used
$$ 	|\mathbb{S}_{d-1}|=\frac{2\pi^{\frac{d}{2}}}{\Gamma\left(\frac{d}{2}\right)}.
$$
Using the change of variables $r^2\rightarrow r$, we find
\[
\int_{\mathbb{R}^d}M[f]\,dv = \rho_f\frac{\Gamma\left(\frac{\gamma}{\gamma-1}\right)}{\Gamma(\frac n2+1)\Gamma\left(\frac{d}{2}\right)}	\int_0^1\left(1-r\right)^{\frac n2} r^{\frac d2-1}\,dr = \rho_f\frac{\Gamma\left(\frac{\gamma}{\gamma-1}\right)}{\Gamma(\frac n2+1)\Gamma\left(\frac{d}{2}\right)}	B\left(\frac d2,\frac n2+1\right).
\]
Here $B(\cdot,\cdot)$ denotes the Beta function:
$$
B\left(p,q\right):=\int_0^1\left(1-r\right)^{q-1} r^{p-1}\,dr,
$$
where $p$ and $q$ are complex number whose real parts are  positive. It is well-known that the Beta function is related to the Gamma function as
$$
B(p,q)=\frac{\Gamma(p)\Gamma(q)}{\Gamma(p+q)}.
$$
This together with the definition of $k$ deduces
	\begin{align*}
\int_{\mathbb{R}^d}M[f]\,dv 	&=\rho_f\frac{\Gamma\left(\frac{\gamma}{\gamma-1}\right)}{\Gamma(\frac n2+1)\Gamma\left(\frac{d}{2}\right)}	\frac{\Gamma\left(\frac d2\right)\Gamma\left(\frac n2+1\right)}{\Gamma\left(\frac{d+n}{2}+1\right)}=\rho_f.
	\end{align*}
Similarly, we also obtain
	\begin{align*} 
	\int_{\mathbb{R}^d}vM[f]\,dv &=c	\int_{\mathbb{R}^d}v\left(\frac{2\gamma}{\gamma-1}\rho_f^{\gamma-1}-|v-u_f|^2\right)^{\frac n2} \mathbf{1}_{|v-u_f|^2\le \frac{2\gamma}{\gamma-1}\rho_f^{\gamma-1}}\,dv\cr
	&= c	\int_{\mathbb{R}^d}(v+u_f)\left(\frac{2\gamma}{\gamma-1}\rho_f^{\gamma-1}-|v|^2\right)^{\frac n2} \mathbf{1}_{|v|^2\le \frac{2\gamma}{\gamma-1}\rho_f^{\gamma-1}}\,dv\cr
	&= cu_f\left(\frac{2\gamma}{\gamma-1}\rho_f^{\gamma-1}\right)^{\frac{n+d}{2}}	\int_{\mathbb{R}^d}\left(1-|v|^2\right)^{\frac n2} \mathbf{1}_{|v|^2\le 1}\,dv\cr
	&= \rho_f u_f.
	\end{align*}
Finally, we estimate 
	\begin{align*} 
	\int_{\mathbb{R}^d}|v|^2M[f]\,dv 
	& = c\int_{\mathbb{R}^d}|v+u_f|^2\left(\frac{2\gamma}{\gamma-1}\rho_f^{\gamma-1}-|v|^2\right)^{\frac n2} \mathbf{1}_{|v|^2\le \frac{2\gamma}{\gamma-1}\rho_f^{\gamma-1}}\,dv\cr
	&  = c|u_f|^2\left(\frac{2\gamma}{\gamma-1}\rho_f^{\gamma-1}\right)^{\frac{n+d}{2}}	\int_{\mathbb{R}^d}\left(1-|v|^2\right)^{\frac n2} \mathbf{1}_{|v|^2\le 1}\,dv\cr
	&\quad +c\left(\frac{2\gamma}{\gamma-1}\rho_f^{\gamma-1}\right)^{\frac{n+d}{2}+1}	\int_{\mathbb{R}^d}|v|^2\left(1-|v|^2\right)^{\frac n2} \mathbf{1}_{|v|^2\le 1}\,dv\cr
	&= \rho_f |u_f|^2 +\frac{2\gamma}{\gamma-1}\rho_f^{\gamma-1}\frac{\Gamma\left(\frac{\gamma}{\gamma-1}\right)}{\pi^{\frac d2}\Gamma(\frac n2+1)}\rho_f	\left|\mathbb{S}_{d-1}\right|\int_0^1\left(1-r^2\right)^{\frac n2} r^{d+1}\,dr\cr
	&=\rho_f |u_f|^2 +\frac{2\gamma}{\gamma-1}\rho_f^{\gamma}\frac{\Gamma\left(\frac{\gamma}{\gamma-1}\right)}{\Gamma(\frac n2+1)\Gamma\left(\frac{d}{2}\right)}B\left(\frac{d}{2}+1,\frac n2+1\right).
	\end{align*}
Note that for positive real numbers $x$ and $y$, the Beta function satisfies  
	$$
	B(x+1,y)=\frac{x}{x+y}B(x,y).
	$$
Using the above relation, we have
	\begin{align*} 
	\int_{\mathbb{R}^d}|v|^2M[f]\,dv &=\rho_f |u_f|^2 +\frac{2\gamma}{\gamma-1}\rho_f^{\gamma}\frac{\Gamma\left(\frac{\gamma}{\gamma-1}\right)}{\Gamma(\frac n2+1)\Gamma\left(\frac{d}{2}\right)}\left\{\frac{d/2}{(d+n)/2+1}B\left(\frac{d}{2},\frac n2+1\right)\right\}\cr
	&=\rho_f |u_f|^2 +d\rho_f^{\gamma}.
	\end{align*}
This completes the proof.
\end{proof}

%
%
%
%
%
%

\section{Proof of Lemma \ref{lipschitz}: positive part function case}\label{app_lem}
In this appendix, we provide the details on the proof of Lemma \ref{lipschitz} when the equilibrium function is given as the positive part function in \eqref{Maxwellian}, i.e. $M[f]$ takes the form of
\begin{equation*}
M [f]=  
c\left(\frac{2\gamma}{\gamma-1}\rho_f^{\gamma-1}-|v-u_f|^2\right)^{n/2}_+,
\end{equation*}
where the constants $c$ and $n$ are given as
\begin{align*}
c=\left(\frac{2\gamma}{\gamma-1}\right)^{-\frac{1}{\gamma-1}}\frac{\Gamma\left(\frac{\gamma}{\gamma-1}\right)}{\pi^{d/2}\Gamma(n/2+1)} \qquad \mbox{and} \qquad  n=\frac{2}{\gamma-1}-d.
\end{align*}
To apply the similar argument used in Section \ref{sec_cau}, we rewrite the equilibrium function in terms of an indicator function as
\begin{equation*}
M [f]=  
c\left(\frac{2\gamma}{\gamma-1}\rho_f^{\gamma-1}-|v-u_f|^2\right)^{n/2} \mathbf{1}_{|v-u_f|^2\le \frac{2\gamma}{\gamma-1}\rho_f^{\gamma-1}}.
\end{equation*}
Through the following two subsections, we provide the Lipschitz continuity of $M[f]$ in $L^1_2(\R^d \times \R^d)$ for $1<\gamma\le\frac{d+4}{d+2}$ when $d \geq 2$ and $1 < \gamma < 3$ when $d=1$. For this, similarly as in Section \ref{sec_cau}, we divide the proof into two cases: $d=1$ and $d\geq 2$.

\subsection{Proof of Lemma \ref{lipschitz} in the case $d=1$ and $1<\gamma < 3$}

In the mono-dimensional case, we decompose $\mathbb{R}$ into three parts as $\mathbb{R}=\sfD_1\cup \sfD_2\cup \sfD_3$, where
	\begin{align*}\begin{split}
	\sfD_1&:=\left\{v\in\mathbb{R}\ :\ |u_f-u_g|> r_f+r_g  \right\},\cr
	\sfD_2&:=\left\{v\in\mathbb{R}\ :\ \left| r_f-r_g\right|\le |u_f-u_g|\le r_f+r_g \right\},\quad \mbox{and}\cr
	\sfD_3&:=\left\{v\in\mathbb{R}\ : \ |u_f-u_g|\le \left|r_f-r_g\right| \right\}
	\end{split}\end{align*}
	with
	$$
	r_f:=\left(\frac{2\gamma}{\gamma-1}\right)^{\frac 12}\rho_f^{\frac{\gamma-1}{2}}\quad \mbox{and} \quad r_g:=\left(\frac{2\gamma}{\gamma-1}\right)^{\frac 12}\rho_g^{\frac{\gamma-1}{2}}.
	$$
 We then split $\|M[f]-M[g]\|_{L^1_2}$ into three terms:
	\begin{align*}
	\int_{\mathbb{R}}(1+v^2)\left|M[f]-M[g]\right| dv=\sum_{i=1}^3\int_{\sfD_i}(1+v^2)\left|M[f]-M[g]\right| dv =: I + II + III.
	\end{align*}
	By symmetry, without loss of generality, we only deal with the case $u_f \le u_g$.
		
		\vspace{.2cm}
		
	$\bullet$ Estimate of $I$: Since the supports of $M[f]$ and $M[g]$ do not intersect on the domain $\sfD_1$, we get
	\begin{align*}
	I=\int_{\mathbb{R}} (1+v^2)M[f]\,dv+\int_{\mathbb{R}} (1+v^2)M[g]\,dv=\left(\rho_f+\rho_fu_f^2+\rho_f^\gamma\right)+\left(\rho_g+\rho_gu_g^2+\rho_g^\gamma\right).
	\end{align*}
	Using the condition of $\sfD_1$, we find
	\begin{align*}
	\rho_f^\gamma+\rho_g^\gamma&=\left(\frac{2\gamma}{\gamma-1}\right)^{-\frac{\gamma}{\gamma-1}}\left(r_f^{\frac{2\gamma}{\gamma-1}}+r_g^{\frac{2\gamma}{\gamma-1}} \right)\le C \left(r_f+r_g \right)^{\frac{2\gamma}{\gamma-1}}\le C\left|u_f-u_g \right|^{\frac{2\gamma}{\gamma-1}},
	\end{align*}
and thus
	\begin{align}\label{I1}\begin{split}
	I&=\left(\rho_f+\rho_g\right)\left(1+|u_f|^2 \right) -(u_f-u_g) \rho_g(u_f+u_g)+\rho_f^\gamma+\rho_g^\gamma\cr
	&\le C\left|u_f-u_g\right|^{\frac{2}{\gamma-1}}\left(1+|u_f|^2 \right)+\left|u_f-u_g\right| \rho_g(u_f+u_g)+C\left|u_f-u_g\right|^{\frac{2\gamma}{\gamma-1}}.
	\end{split}\end{align}
	Since $1<\gamma < 3 $, this together with the bound assumptions \eqref{bounds2} gives
	\begin{align*}
	I\le C|u_f-u_g|.
	\end{align*}
	
		\vspace{.2cm}
	
	$\bullet$ Estimate of $II$: In this case, the supports of $M[f]$ and $M[g]$ partially intersect, see Fig. \ref{domain_21}. We estimate
	\begin{align*}
	II&=\int^{u_g-r_g}_{u_f-r_f} (1+v^2)M[f]\,dv+\int^{u_g+r_g}_{u_f+r_f} (1+v^2)M[g]\,dv\cr
	&\quad +c\int_{u_g-r_g}^{u_f+r_f} (1+v^2)\left|\left(r_f^2-|v-u_f|^2\right)^{\frac n2}-\left(r_g^2-|v-u_g|^2\right)^{\frac n2}\right|\,dv\cr
	&=: II_1+II_2+II_3.
	\end{align*}
		
	\begin{figure}[!h]
		\begin{center}
			\includegraphics[scale=0.4]{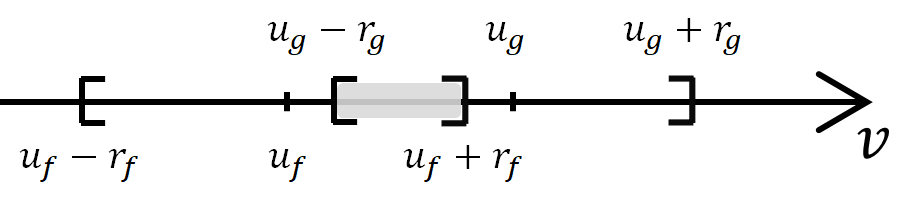}
			\caption{
Illustration of the domain $\sfD_2$ 
}
\label{domain_21}
		\end{center}
	\end{figure} 
	
	-- Estimate of $II_1$: Since $II_1$ is defined within the support of $M[f]$, we obtain
	\begin{align*}
	II_1&=c\int^{u_g-r_g}_{u_f-r_f} (1+v^2)\left(r_f^2-|v-u_f|^2\right)^{\frac n2}  dv\cr
	&\le cr_f^n\int^{u_g-r_g}_{u_f-r_f} (1+v^2) \,dv\cr
	&= cr_f^n (u_g-u_f+r_f-r_g)\left\{1+\frac 13\left((u_g-r_g)^2+(u_g-r_g)(u_f-r_g)+(u_f-r_g)^2\right)  \right\}.
	\end{align*}
	On the other hand, by the mean value theorem, we get
	\begin{align*}
	r_f-r_g&=\left(\frac{2\gamma}{\gamma-1}\right)^{\frac 12}\left( \rho_f^{\frac{\gamma-1}{2}}-\rho_g^{\frac{\gamma-1}{2}}\right)=\left(\frac{2\gamma}{\gamma-1}\right)^{\frac 12} \frac{\gamma-1}{2}(\rho_f-\rho_g)\int_0^1   \left( \theta\rho_f+(1-\theta)\rho_g\right)^{\frac{\gamma-1}{2}-1} d\theta
	\end{align*}
	which, together with the bound assumptions \eqref{bounds2} gives
	\begin{equation}\label{r rho}
	\left|r_f-r_g\right|\le C \left| \rho_f-\rho_g\right|.
	\end{equation}
	Thus,  it follows from  \eqref{bounds2}  that
	\begin{align*}
	II_1&\le cr_f^n \left(|u_f-u_g|+|r_f-r_g|\right)\left\{1+\frac 13\left((u_g-r_g)^2+(u_g-r_g)(u_f-r_g)+(u_f-r_g)^2\right)  \right\}\cr
	&\le C\left(|\rho_f-\rho_g|+|u_f-u_g|\right).
	\end{align*}
	
 \vspace{.2cm}
		
	-- Estimate of $II_2$: In the same manner as in $II_1$, we find
	\begin{align*}
	II_2&\le  	 cr_g^n \left(|u_f-u_g|+|r_f-r_g|\right)\left\{1+\frac 13\left((u_g+r_g)^2+(u_g+r_g)(u_f+r_g)+(u_f+r_g)^2\right)  \right\} \cr
	&\le C\left(|\rho_f-\rho_g|+|u_f-u_g|\right).
	\end{align*}
	
 \vspace{.2cm}
		
	-- Estimate of $II_3$: By the mean value theorem, we deduce
	\begin{align*}
	&\left|\left(r_f^2-|v-u_f|^2\right)^{\frac n2}-\left(r_g^2-|v-u_g|^2\right)^{\frac n2}\right|\cr
	&\quad =\frac n2\left|\left(r_f^2-|v-u_f|^2-r_g^2+|v-u_g|^2\right)\int_0^1  \left\{\theta\left(r_f^2-|v-u_f|^2 \right)+(1-\theta)\left(r_g^2-|v-u_g|^2\right) \right\}^{\frac n2-1} d\theta \right|\cr
	&\quad \le\frac n2\left|\left(r_f-r_g\right)\left(r_f+r_g\right)\int_0^1  \left\{\theta\left(r_f^2-|v-u_f|^2 \right)+(1-\theta)\left(r_g^2-|v-u_g|^2\right) \right\}^{\frac n2-1} d\theta \right|\cr
	&\qquad +\frac n2\left|(u_f-u_g)\left(2v-u_f-u_g\right)\int_0^1  \left\{\theta\left(r_f^2-|v-u_f|^2 \right)+(1-\theta)\left(r_g^2-|v-u_g|^2\right) \right\}^{\frac n2-1}  d\theta \right|. 
	\end{align*}
	Thus we use \eqref{bounds2} to estimate
	\begin{align}\label{1d I230}\begin{split}
	&\left|\left(r_f^2-|v-u_f|^2\right)^{\frac n2}-\left(r_g^2-|v-u_g|^2\right)^{\frac n2}\right|\cr
	&\quad \le C\left|r_f-r_g\right|\left|\int_0^1  \left\{\theta\left(r_f^2-|v-u_f|^2 \right)+(1-\theta)\left(r_g^2-|v-u_g|^2\right) \right\}^{\frac n2-1}  d\theta\right| \cr
	&\qquad +C\left|u_f-u_g\right|\left(1+|v|\right)\left|\int_0^1  \left\{\theta\left(r_f^2-|v-u_f|^2 \right)+(1-\theta)\left(r_g^2-|v-u_g|^2\right) \right\}^{\frac n2-1} d\theta\right| .
	\end{split}\end{align}
	Note that the function $\theta(r_f^2-|v-u_f|^2 )+(1-\theta)\left(r_g^2-|v-u_g|^2\right)$
	is positive for any $v$ belonging to the intersection of the supports of $M[f]$ and $M[g]$. Thus we can see that in \eqref{1d I230}, there is no singularity on the finite interval $ [u_g-r_g,u_f+r_f]$ even though $n/2<1$. Since $u_h$ and $r_h$ for $h \in \{f,g\}$ are bounded by \eqref{bounds2}, we have
	\begin{align}\label{1d I231}
	\begin{aligned}
	II_3&\le C\left|r_f-r_g\right|\int_{u_g-r_g}^{u_f+r_f} (1+v^2) \int_0^1  \left\{\theta\left(r_f^2-|v-u_f|^2 \right)+(1-\theta)\left(r_g^2-|v-u_g|^2\right) \right\}^{\frac n2-1} d\theta \,dv\cr
	&\quad +C|u_f-u_g|\int_{u_g-r_g}^{u_f+r_f} (1+v^2)^2 \int_0^1\left\{\theta\left(r_f^2-|v-u_f|^2 \right)+(1-\theta)\left(r_g^2-|v-u_g|^2\right) \right\}^{\frac n2-1} d\theta \,dv\cr
	&\le C\left(\left|r_f-r_g\right| +|u_f-u_g|\right)\cr
	&\le C\left(|\rho_f-\rho_g|+|u_f-u_g|\right)
	\end{aligned}
	\end{align}
	due to  \eqref{r rho}.
	Combining all of the above estimates yields
	\[
	II \leq C\left(|\rho_f-\rho_g|+|u_f-u_g|\right)
	\]
	for some $C>0$.
	
	\vspace{.2cm}
	
	$\bullet$ Estimate of $III$: To avoid the repetition of estimates, we only deal with the case where the support of $M[f]$ is completely contained within that of $M[g]$, i.e. $[u_f-r_f,u_f+r_f]\subseteq [u_g-r_g,u_g+r_g]$. In this case, we observe
\bq\label{est_III}
	III=\int_{u_g-r_g}^{u_f-r_f} (1+v^2) M[g]\,dv+	\int^{u_g+r_g}_{u_f+r_f} (1+v^2)M[g]\,dv +	\int_{u_f-r_f}^{u_f+r_f} (1+v^2)\left|M[f]-M[g]\right|dv.
\eq
	By \eqref{bounds2}, the first two terms can be estimated as
	\begin{align*}
	&\int_{u_g-r_g}^{u_f-r_f} (1+v^2) M[g]\,dv+\int^{u_g+r_g}_{u_f+r_f} (1+v^2)M[g]\,dv\cr
	&\quad \le c\int_{u_g-r_g}^{u_f-r_f} (1+v^2) \left(\frac{2\gamma}{\gamma-1}\rho_g^{\gamma-1}\right)^{n/2}\,dv+c\int^{u_g+r_g}_{u_f+r_f} (1+v^2)\left(\frac{2\gamma}{\gamma-1}\rho_g^{\gamma-1}\right)^{n/2}dv\cr
	&\quad \le C\left(|u_f-u_g|+|r_f-r_g| \right)
	\end{align*}
	which, combined with \eqref{r rho}, gives
	$$
	\int_{u_g-r_g}^{u_f-r_f} (1+v^2) M[g]\,dv+	\int^{u_g+r_g}_{u_f+r_f} (1+v^2)M[g]\,dv\le C\left(|\rho_f-\rho_g|+|u_f-u_g|\right).
	$$
	In the same manner as in the case $II_3$, we apply \eqref{1d I231} to the last term on the right hand side of \eqref{est_III} to obtain
	\begin{align*}
	\int_{u_f-r_f}^{u_f+r_f} (1+v^2)\left|M[f]-M[g]\right| dv&\le C	(\left|r_f-r_g\right|+\left|u_f-u_g\right|)\le C\left(|\rho_f-\rho_g|+|u_f-u_g|\right).
	\end{align*}
This completes the proof.

\subsection{Proof of Lemma \ref{lipschitz} in the case $d\geq 2$ and $1<\gamma\le\frac{d+4}{d+2}$}

We now consider the multi-dimensional cases $d\ge 2$, and in this case, we assume $\gamma\in (1,\frac{d+4}{d+2}]$. We begin by introducing several notations:
	\begin{align*}
	r_f&=\left(\frac{2\gamma}{\gamma-1}\right)^{\frac 12}\rho_f^{\frac{\gamma-1}{2}},\qquad r_g=\left(\frac{2\gamma}{\gamma-1}\right)^{\frac 12}\rho_g^{\frac{\gamma-1}{2}},\qquad U_{f,g}=|u_f-u_g|,\cr
	&\theta_f=\arccos \left(\frac{r_f^2+U_{f,g}^2-r_g^2}{2r_f U_{f,g}}\right),\quad \mbox{and} \quad \tilde{\theta}=\arccos\left(\frac{r_f^2-r_g^2-U_{f,g}^2}{2r_g U_{f,g}}\right).
	\end{align*}
	Using these newly defined functions, similarly as in Section \ref{ssec_bdy}, we split the domain $\mathbb{R}^d$ into four cases:
	\begin{align*}
	\sfD_1&=\left\{v\in\mathbb{R}^d : |u_f-u_g|> r_f+r_g  \right\},\cr
	\sfD_2&=\left\{v\in\mathbb{R}^d :  \left|r_f-r_g \right|\le |u_f-u_g|\le r_f+r_g\quad \text{and}\quad |u_f-u_g|^2> \left|r_f^2-r_g^2 \right|  \right\},\cr
	\sfD_3&=\left\{v\in\mathbb{R}^d : |r_f-r_g|\le |u_f-u_g|\le r_f+r_g\quad \text{and}\quad |u_f-u_g|^2\le\left|r_f^2-r_g^2\right|  \right\},\quad \mbox{and}\cr
	\sfD_4&=\left\{v\in\mathbb{R}^d :  |u_f-u_g|\le |r_f-r_g| \right\}.
	\end{align*}	
	The proof is similar to that of the case $\gamma = \frac{d+2}d$, but slightly simpler since for $\gamma\le\frac{d+4}{d+2}$, we can make use of the mean value theorem for the equilibrium function  as
	\begin{align*}
	&\left|\left(r_f^2-|v-u_f|^2\right)^{\frac n2}-\left(r_g^2-|v-u_g|^2\right)^{\frac n2}\right|\cr
	&\quad =\frac n2\left|\left(r_f^2-|v-u_f|^2-r_g^2+|v-u_g|^2\right)\int_0^1  \left\{\theta\left(r_f^2-|v-u_f|^2 \right)+(1-\theta)\left(r_g^2-|v-u_g|^2\right) \right\}^{\frac n2-1} \,d\theta \right|\cr
	&\quad \le\frac n2\left|\left(r_f-r_g\right)\left(r_f+r_g\right)\int_0^1  \left\{\theta\left(r_f^2-|v-u_f|^2 \right)+(1-\theta)\left(r_g^2-|v-u_g|^2\right) \right\}^{\frac n2-1} \,d\theta \right|\cr
	&\qquad +\frac n2\left|(u_f-u_g)\left(2v-u_f-u_g\right)\int_0^1  \left\{\theta\left(r_f^2-|v-u_f|^2 \right)+(1-\theta)\left(r_g^2-|v-u_g|^2\right) \right\}^{\frac n2-1} \,d\theta \right|. 
	\end{align*}
	This, combined with \eqref{bounds2} leads to
	\begin{align}\label{1d I23}\begin{split}
	&\left|\left(r_f^2-|v-u_f|^2\right)^{\frac n2}-\left(r_g^2-|v-u_g|^2\right)^{\frac n2}\right|\cr
	&\quad \le C\left|r_f-r_g\right|\left|\int_0^1  \left\{\theta\left(r_f^2-|v-u_f|^2 \right)+(1-\theta)\left(r_g^2-|v-u_g|^2\right) \right\}^{\frac n2-1} \,d\theta\right| \cr
	&\qquad +C\left|u_f-u_g\right|\left(1+|v|\right)\left|\int_0^1  \left\{\theta\left(r_f^2-|v-u_f|^2 \right)+(1-\theta)\left(r_g^2-|v-u_g|^2\right) \right\}^{\frac n2-1} \,d\theta\right| 
	\end{split}\end{align}
	which will be fruitfully used in this proof.
	Observe that
	\begin{align*}
		\intr(1+|v|^2)\left|M[f]-M[g]\right|\,dv=	\sum_{i=1}^4\int_{\sfD_i}(1+|v|^2)\left|M[f]-M[g]\right|dv=: \widetilde I + \widetilde{II} + \widetilde{III} + \widetilde{IV}.
	\end{align*}
	We then estimate it separately and only consider the case $r_f \le r_g$. 
	
\vspace{.2cm}

	$\bullet$ Estimate of $\widetilde{I}$: It follows from the velocity-moment estimates in \eqref{moment_comp} that
	\begin{align*}
	\widetilde{I}&=\int_{\mathbb{R}^d} (1+|v|^2)M[f]\,dv+\int_{\mathbb{R}^d} (1+|v|^2)M[g]\,dv		\cr
	&=\left(\rho_f+\rho_f |u_f|^2 +d\rho_f^{\gamma}\right)+\Big(\rho_g+\rho_g |u_g|^2 +d\rho_g^{\gamma}\Big)\cr
	&=(\rho_f+\rho_g)\left(1+|u_f|^2\right)+\rho_g \left(u_g+u_f\right)\cdot\left(u_g-u_f\right)+d\left(\rho_f^\gamma+\rho_g^\gamma\right).
	\end{align*}
	In the same manner as \eqref{I1}, we deduce
	$$
	\widetilde{I}\le C|u_f-u_g|.
	$$
	
	\vspace{.2cm}
	
	$\bullet$ Estimate of $\widetilde{II}$: Let us denote by $\mathbb{B}_f$ and $\mathbb{B}_g$ the supports of $M[f]$ and $M[g]$ respectively. We then have
	\begin{align*}
	\widetilde{II}=\int_{\mathbb{B}_f \setminus \mathbb{B}_g} (1+|v|^2)M[f]\,dv+\int_{\mathbb{B}_g\setminus\mathbb{B}_f} (1+|v|^2)M[g]\,dv +\int_{\mathbb{B}_f\cap\mathbb{B}_g} (1+|v|^2)\left|M[f]-M[g]\right| dv.
	\end{align*}
	We only deal with the first and third terms on the right hand side of the above for simplicity. For the first term, we translate $u_g$ to the origin. Then $\mathbb{B}_f \setminus \mathbb{B}_g$ can be described in spherical coordinates as either 
	\begin{align*}
	\mathbb{B}_f \setminus \mathbb{B}_g&=\left\{ (r,\theta,\varphi_1,\dots,\varphi_{d-2}) : r_g\le r\le r_1,\ 0\le \theta\le \theta_1,\ 0\le \varphi_1,\dots,\varphi_{d-3}\le \pi,\  0\le \varphi_{d-2}< 2\pi \right\}\cr
	&\ \quad \cup \left\{ (r,\theta,\varphi_1,\dots,\varphi_{d-2}) : r_2\le r\le r_1,\ \theta_1\le \theta\le \theta_2,\ 0\le \varphi_1,\dots,\varphi_{d-3}\le \pi,\  0\le \varphi_{d-2}< 2\pi \right\}\cr
	&=: \mathbb{B}_1\cup \mathbb{B}_2
	\end{align*}
	or
	\begin{align*}
	\mathbb{B}_f \setminus \mathbb{B}_g&=\left\{ (r,\theta,\varphi_1,\dots,\varphi_{d-2}) : r_g\le r\le r_1,\ 0\le \theta\le \theta_1,\ 0\le \varphi_1,\dots,\varphi_{d-3}\le \pi,\  0\le \varphi_{d-2}< 2\pi \right\}\cr
	&=: \mathbb{B}_3,
	\end{align*}
	where $r_1, r_{2}, \theta_1,$ and $ \theta_2 $ are given as
	\begin{align*}
	r_1&:=\sqrt{r_f^2-U_{f,g}^2+U_{f,g}^2\cos^2\theta}+U_{f,g}\cos\theta,\quad r_{2}:=U_{f,g}\cos\theta-\sqrt{r_f^2-U_{f,g}^2+U_{f,g}^2\cos^2\theta},\cr  &\qquad\quad\qquad\cos\theta_1:=\frac{r_g^2+U_{f,g}^2-r_f^2}{2r_gU_{f,g}},\quad \mbox{and} \quad \cos\theta_2:=\frac{\sqrt{U_{f,g}^2-r_f^2}}{U_{f,g}},
	\end{align*}
	respectively (see Fig. \ref{domains2}).
	 \begin{figure}[!h]
		\begin{center}
			\includegraphics[scale=0.45]{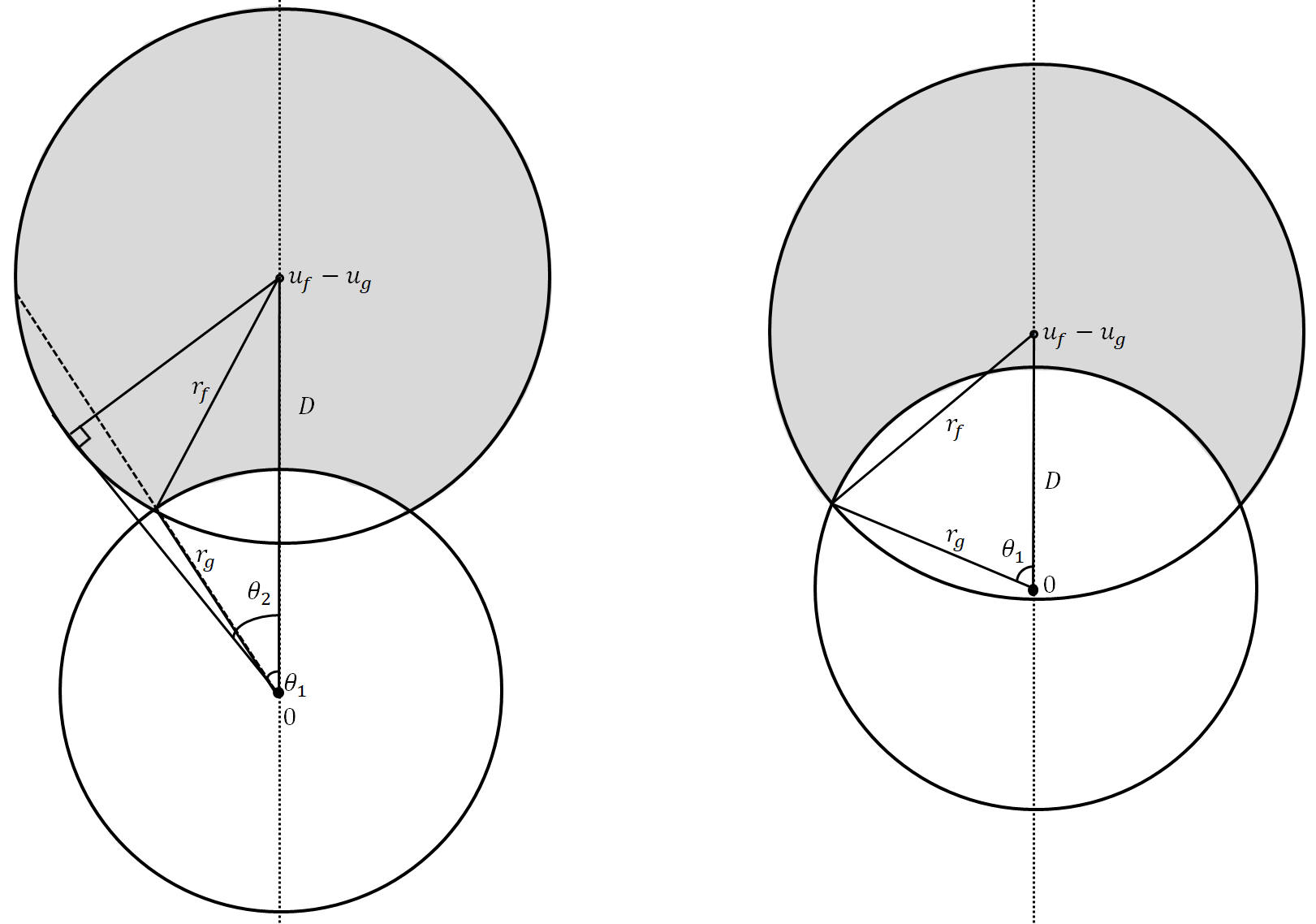}
			\caption{
				Illustrations of the domains $\mathbb{B}_1\cup \mathbb{B}_2$ \& $\mathbb{B}_3$
			}
			\label{domains2}
		\end{center}
	\end{figure} 
	 By \eqref{bounds2}, the case of $\mathbb{B}_1$ can be estimated as
	\begin{align*}
	&\int_{\mathbb{B}_1}(1+|v|^2)M[f]\,dv\cr
	&\quad =\int_0^{2\pi}\int_0^\pi\cdots\int_0^{\theta_1}\int_{r_g}^{r_1}\left(1+|u_g|^2+2v\cdot u_g+r^2\right)\left(\frac{2\gamma}{\gamma-1}\rho_f^{\gamma-1}-\left(r^2+2v\cdot (u_g-u_f)+|u_g-u_f|^2\right)\right)^{n/2}\cr
	&\hspace{6cm} \times r^{d-1}(\sin^{d-2}\theta) (\sin^{d-3}\varphi_{1})\cdots (\sin\varphi_{d-3})\,drd\theta d\varphi_1\cdots d\varphi_{d-2} \cr
	&\quad \le C\int_0^{2\pi}\int_0^\pi\cdots\int_0^{\theta_1}(r_1-r_g)\,d\theta d\varphi_1\cdots d\varphi_{d-2} \cr
	&\quad \le C(r_g-r_f+U_{f,g}).
	\end{align*}
	In the last line, we used the fact that  $r_1$ is decreasing in $\theta\in [0,\pi ]$. Thus it follows from \eqref{r rho} that
	$$
	\int_{\mathbb{B}_1}(1+|v|^2)M[f]\,dv \le C(|\rho_f-\rho_g|+|u_f-u_g|).
	$$
	For $\mathbb{B}_2$, we obtain
	\begin{align*}
	&\int_{\mathbb{B}_2}(1+|v|^2)M[f]\,dv\cr
	&\quad =\int_0^{2\pi}\int_0^\pi\cdots\int^{\theta_2}_{\theta_1}\int_{r_{21}}^{r_{22}}\left(1+|u_g|^2+2v\cdot u_g+r^2\right)\left(\frac{2\gamma}{\gamma-1}\rho_f^{\gamma-1}-\left(r^2+2v\cdot (u_g-u_f)+|u_g-u_f|^2\right)\right)^{n/2}\cr
	&\hspace{6cm}\times r^{d-1}(\sin^{d-2}\theta)( \sin^{d-3}\varphi_{1})\cdots (\sin\varphi_{d-3})\,drd\theta d\varphi_1\cdots d\varphi_{d-2} \cr
	&\quad \le C(\theta_2-\theta_1).
	\end{align*}
	By \eqref{arc}, we get
	\begin{align*}
	|\theta_2-\theta_1|&=\left|\arcsin\left(\frac{r_g^2+U_{f,g}^2-r_f^2}{2r_gU_{f,g}} \right) -\arcsin\left(\frac{\sqrt{U_{f,g}^2-r_f^2}}{U_{f,g}}\right) \right|\cr
	&\le C\left(\frac{r_g^2+U_{f,g}^2-r_f^2}{2r_gU_{f,g}}+\frac{\sqrt{U_{f,g}^2-r_f^2}}{U_{f,g}}\right)\cr
	&\le C U_{f,g},
	\end{align*}
	where we used \eqref{bounds2} and the condition of $\sfD_2$. This gives
	\begin{align*}
	\int_{\mathbb{B}_2}(1+|v|^2)M[f]\,dv\le C|u_f-u_g|.
	\end{align*}
	Since $\mathbb{B}_3$ can be handled in the same manner as $\mathbb{B}_1$, we omit it. Finally, we use \eqref{1d I23} to estimate the third term  as
	\begin{align*}
	&\int_{ \mathbb{B}_f \cap\mathbb{B}_g} (1+|v|^2)\left|M[f]-M[g]\right| dv\cr
	&\quad \le C\int_{\mathbb{B}_f\cap \mathbb{B}_g}\int_0^1\left|r_f-r_g\right|(1+|v|^2)  \left\{\theta\left(r_f^2-|v-u_f|^2 \right)+(1-\theta)\left(r_g^2-|v-u_g|^2\right) \right\}^{\frac n2-1} dvd\theta \cr
	&\qquad +C\int_{\mathbb{B}_f\cap \mathbb{B}_g}\int_0^1\left|u_f-u_g\right|\left(1+|v|^2\right)^2   \left\{\theta\left(r_f^2-|v-u_f|^2 \right)+(1-\theta)\left(r_g^2-|v-u_g|^2\right) \right\}^{\frac n2-1} dvd\theta.
	\end{align*}
	By \eqref{bounds2}, the integrands are bounded on $\mathbb{B}_f\cap\mathbb{B}_g$  and  $|\mathbb{B}_f\cap\mathbb{B}_g|$ is finite as well, so we get
	\begin{align}\label{meanvalue}\begin{split}
	\int_{ \mathbb{B}_f \cap\mathbb{B}_g} (1+|v|^2)\left|M[f]-M[g]\right|dv \le C\left|r_f-r_g\right|  +\left|u_f-u_g\right| \le C\left(|\rho_f-\rho_g|+|u_f-u_g|\right).
	\end{split}\end{align}

\vspace{.2cm}

	$\bullet$ Estimate of $\widetilde{III}$: Similarly to the case $\widetilde{II}$,  we have from \eqref{meanvalue} that
	\begin{align*}
	\widetilde{III}\le \int_{\mathbb{B}_f\setminus\mathbb{B}_g} (1+|v|^2)M[f]\,dv+\int_{\mathbb{B}_g\setminus\mathbb{B}_f} (1+|v|^2)M[g]\,dv  +C\left(|\rho_f-\rho_g|+|u_f-u_g|\right).
	\end{align*}
	When we translate $u_g$ to the origin, $\mathbb{B}_f\setminus\mathbb{B}_g$ is expressed in spherical coordinates as
	\begin{align*}
	\mathbb{B}_f\setminus\mathbb{B}_g&=\left\{ (r,\theta,\varphi_1,\dots,\varphi_{d-2}) : r_g\le r\le \tilde{r},\ 0\le \theta\le \tilde{\theta},\ 0\le \varphi_1,\dots,\varphi_{d-3}\le \pi,\  0\le \varphi_{d-2}< 2\pi \right\}
	\end{align*}
	with
	\begin{align*}
	\tilde{r}=&\sqrt{r_f^2-U_{f,g}^2+U_{f,g}^2\cos^2\theta}+U_{f,g}\cos\theta\quad \mbox{and} \quad \cos\tilde{\theta}=\frac{r_g^2+U_{f,g}^2-r_f^2}{2r_gU_{f,g}}.
	\end{align*}
	By the same argument as in $\mathbb{B}_1$ of $\widetilde{II}$, we can conclude that
	$$
	\widetilde{III}\le C\left(|\rho_f-\rho_g|+|u_f-u_g|\right).
	$$  
	
	\vspace{.2cm}
	
	$\bullet$ Estimate of $\widetilde{IV}$: To avoid the repetition of estimates, we only provide the result in the case that $\mathbb{B}_g$ is contained within  $\mathbb{B}_f$. Then we have
	\begin{align*}
	\widetilde{IV}&=\int_{ \mathbb{B}_g} (1+|v|^2)\left|M[f]-M[g]\right| dv+\int_{\mathbb{B}_f-\mathbb{B}_g} (1+|v|^2) M[f]\,dv.
	\end{align*} 
	In the same manner as \eqref{meanvalue}, the first term can be estimated as
	\begin{align*}
	\int_{\mathbb{B}_g} (1+|v|^2)\left|M[f]-M[g]\right| dv\le C\left(|\rho_f-\rho_g|+|u_f-u_g|\right).
	\end{align*}
	For the second term, we translate $u_g$ to the origin to express $\mathbb{B}_f\setminus\mathbb{B}_g$ in spherical coordinates as 
	$$
	\mathbb{B}_f\setminus\mathbb{B}_g=\left\{ (r,\theta,\varphi_1,\dots,\varphi_{d-2}) : r_g\le r\le \tilde{r},\ 0\le \theta\le \pi,\ 0\le \varphi_1,\dots,\varphi_{d-3}\le \pi,\  0\le \varphi_{d-2}< 2\pi \right\},
	$$
	where $\tilde{r} > 0$ is given by
	$$
	\tilde{r}:=-U_{f,g}\cos\theta +\sqrt{r_f^2-U_{f,g}^2+U_{f,g}^2\cos^2\theta}.
	$$
	Then, it follows from \eqref{bounds2} that
	\begin{align*}
	&\int_{\mathbb{B}_f-\mathbb{B}_g} (1+|v|^2) M[f]\,dv\cr
	&\quad =\int_0^{2\pi}\int_0^\pi\cdots\int_0^{\pi}\int_{r_g}^{\tilde{r}}\left(1+|u_g|^2+2v\cdot u_g+r^2\right)\left(\frac{2\gamma}{\gamma-1}\rho_f^{\gamma-1}-\left(r^2+2v\cdot (u_g-u_f)+|u_g-u_f|^2\right)\right)^{n/2}\cr
	&\hspace{6cm} \times r^{d-1}(\sin^{d-2}\theta)( \sin^{d-3}\varphi_{1})\cdots (\sin\varphi_{d-3})\,drd\theta d\varphi_1\cdots d\varphi_{d-2} \cr
	&\quad \le C\int_0^{2\pi}\int_0^\pi\cdots\int_0^{\pi}(\tilde{r}-r_g)\,d\theta d\varphi_1\cdots d\varphi_{d-2}\cr
	&\quad \le C\left(U_{f,g}+r_f-r_g \right).
	\end{align*}
	In the last line, we used the fact that $\tilde{r}$ is increasing in $\theta\in [0,\pi]$. This together with \eqref{r rho} yields
	$$
	\int_{\mathbb{B}_f\setminus\mathbb{B}_g} (1+|v|^2) M[f]\,dv \le C\left(|\rho_f-\rho_g|+|u_f-u_g|\right).
	$$
This completes the proof.




\section*{Acknowledgments}
Y.-P. Choi and  B.-H. Hwang  were supported by National Research Foundation of Korea(NRF) grant funded by the Korea government(MSIP) (No. 2022R1A2C1002820).

%
%
%
%

\end{document}